\documentclass[12pt]{article}

\newcommand{\version}{October 29, 2024}

\usepackage{pictex}
\usepackage{mathrsfs}
\usepackage{amssymb}

\usepackage{algorithm}
\usepackage[noend]{algpseudocode}

\usepackage{enumitem}
\setlist{nosep}
\setlist[enumerate,1]{label=(\arabic*)}

\title{Every $3$-connected $\{\claw,\Gamma_3\}$-free graph is 
        Hamilton-connected$^{1}$}

\author{Adam Kabela \and Zden\v{e}k Ryj\'a\v{c}ek \and 
        M\'aria Skyvov\'a \and Petr Vr\'ana}
      
\date{\version}

\newcounter{mathitem}
\newenvironment{mathitem}
  {\begin{list}{{$(\roman{mathitem})$}}{
   \setcounter{mathitem}{0}
   \usecounter{mathitem}
   \setlength{\topsep}{0pt plus 2pt minus 0pt}
   \setlength{\parskip}{0pt plus 2pt minus 0pt}
   \setlength{\partopsep}{0pt plus 2pt minus 0pt}
   \setlength{\parsep}{0pt plus 2pt minus 0pt}
   \setlength{\leftmargin}{35pt}
   \setlength{\itemsep}{0pt plus 2pt minus 0pt}}}
  {\end{list}}

\newcounter{mylist}
\newenvironment{mylist}
  {\begin{list}{{$\roman{mylist}$}}{
   \setlength{\topsep}{7pt plus 2pt minus 0pt}
   \setlength{\parsep}{3pt plus 2pt minus 0pt}
   \setlength{\leftmargin}{12pt}
   \setlength{\labelwidth}{-6pt}
   }
   }
  {\end{list}}

\newcounter{prostredi}
\def\theprostredi{\arabic{prostredi}}
\def\vejde#1{\unskip
\nobreak\hfill\penalty50\hskip1em\hbox{}\nobreak\hfill
\hbox{#1}}
\newenvironment{theorem}{\par\bigskip\noindent%
\refstepcounter{prostredi}{\bf Theorem \theprostredi.}\quad\bgroup\sl }
{\egroup\par\bigskip\endtrivlist}%
\newenvironment{lemma}{\par\bigskip\noindent%
\refstepcounter{prostredi}{\bf Lemma \theprostredi.}\quad\bgroup\sl }
{\egroup\par\bigskip\endtrivlist}%
\newenvironment{proof}{\par
\noindent%
{\bf Proof.}\quad\bgroup}
{\egroup\vejde{\rule{2.5mm}{2.5mm}}\par\bigskip\endtrivlist}%
\newenvironment{proofbt}{\par
\noindent%
{\bf Proof}\bgroup}
{\egroup\vejde{\rule{2.5mm}{2.5mm}}\par\bigskip\endtrivlist}%
\newenvironment{corollary}{\par\bigskip\noindent%
\refstepcounter{prostredi}{\bf Corollary
\theprostredi.}\quad\bgroup\sl }
{\egroup\par\bigskip\endtrivlist}%
\newcounter{prostrclaim}
\def\theprostrclaim{\arabic{prostrclaim}}
\def\vejde#1{\unskip
\nobreak\hfill\penalty50\hskip1em\hbox{}\nobreak\hfill
\hbox{#1}}
\newenvironment{claim}{\par\bigskip\noindent%
\refstepcounter{prostrclaim}{\underline{Claim \theprostrclaim.}}\quad\bgroup\sl }
{\egroup\par\bigskip\endtrivlist}%
\newenvironment{proofcl}{\par
\noindent%
{\underline{Proof.}}\quad\bgroup}
{\egroup\vejde{$\Box$}\par\bigskip\endtrivlist}%
\newenvironment{proofclbt}{\par
\noindent%
{\underline{Proof~}}\bgroup}
{\egroup\vejde{
$\Box$}\par\bigskip\endtrivlist}%

\def\vejde#1{\unskip
\nobreak\hfill\penalty50\hskip1em\hbox{}\nobreak\hfill
\hbox{#1}}
\newenvironment{claimbc}{\par\bigskip\noindent%
{\underline{Claim.}}\quad\bgroup\sl }
{\egroup\par\bigskip\endtrivlist}%

\newcounter{prostralph}
\def\theprostralph{\Alph{prostralph}}
\def\vejde#1{\unskip
\nobreak\hfill\penalty50\hskip1em\hbox{}\nobreak\hfill
\hbox{#1}}
\newenvironment{theoremAcite}[1]{\par\bigskip\noindent%
\refstepcounter{prostralph}{\bf Theorem
\theprostralph{} {#1}.}\quad\bgroup\sl }
{\egroup\par\bigskip\endtrivlist}%
\newenvironment{lemmaAcite}[1]{\par\bigskip\noindent%
\refstepcounter{prostralph}{\bf Lemma
\theprostralph{} {#1}.}\quad\bgroup\sl }
{\egroup\par\bigskip\endtrivlist}%
\newcommand{\noi}{\noindent}
\newcommand{\ld}{\ldots}
\newcommand{\iso}{\simeq}
\newcommand{\niso}{\not \simeq}
\newcommand{\sm}{\setminus}
\newcommand{\cl}{{\rm cl}}
\newcommand{\claw}{K_{1,3}}
\newcommand{\Gt}{\Gamma_3}
\newcommand{\Gi}{\Gamma_i}
\newcommand{\co}{{\rm co}}

\newcommand{\la}{\langle}
\newcommand{\lab}{\langle \{}
\newcommand{\ra}{\rangle}
\newcommand{\rag}{\rangle _G}

\newcommand{\rabg}{\} \rangle _G}

\newcommand{\cF}{{\cal F}}
\newcommand{\cG}{{\cal G}}
\newcommand{\cW}{{\cal W}}
\newcommand{\dvW}{\mathbb{W}}
\newcommand{\dvK}{\mathbb{K}}

\newcommand{\bG}{\bar{G}}
\newcommand{\bq}{\bar{q}}
\newcommand{\tH}{\tilde{H}}
\newcommand{\tF}{\tilde{F}}
\newcommand{\Gst}{G^{^*}}
\newcommand{\Gstx}{G^{^*}_x}
\newcommand{\bp}{\beginpicture}
\newcommand{\ep}{\endpicture}
\newcommand{\dist}{\mbox{dist}}
\newcommand{\Lp}{L^{-1}}
\newcommand{\Int}{{\rm Int}}

\newcommand{\indsub}{\stackrel{\mbox{\tiny IND}}{\subset}}

\newcommand{\bs}{\bigskip}
\newcommand{\bsm}{\vspace{-4mm}}
\newcommand{\ms}{\medskip}
\newcommand{\ssk}{\smallskip}

\newcommand{\Aa}{\mathcal A}
\newcommand{\Bb}{\mathcal B}
\newcommand{\Mm}{\mathcal M}
\newcommand{\AND}{\textbf{and }}
\def\d{\mathop{\rm d}\nolimits}

\begin{document}
\maketitle

\footnotetext[1]{All authors are affiliated with the Department of Mathematics; 
European Centre of Excellence NTIS - New Technologies for the Information 
Society, University of West Bohemia,  Univerzitn\'{\i}~8, 301 00 Pilsen, 
Czech Republic.
E-mails: {\tt $\{$kabela,ryjacek,vranap$\}$@kma.zcu.cz, mskyvova@ntis.zcu.cz.} 
The research was supported by project GA20-09525S of the Czech Science 
Foundation.}

\begin{abstract}
\noi
We show that every $3$-connected $\{\claw,\Gamma_3\}$-free graph is 
Hamilton-connected, where $\Gamma_3$ is the graph obtained by joining 
two vertex-disjoint triangles with a path of length~$3$.
This resolves one of the two last open cases in the characterization of
pairs of connected forbidden subgraphs implying Hamilton-connectedness.
The proof is based on a new closure technique, developed in a previous paper,
and on a structural analysis of small subgraphs, cycles and paths in line 
graphs of multigraphs.
The most technical steps of the analysis are computer-assisted.

\ms

\noi
\end{abstract}

\section{Terminology and notation}
\label{sec-notation}

We generally follow the most common graph-theoretical notation and terminology, 
and for notations and concepts not defined here we refer the reader to \cite{BM08}.
Specifically, by a {\em graph} we always mean a simple finite undirected graph;
whenever we admit multiple edges, we always speak about a {\em multigraph}.
If $G$ is a multigraph and $x_1x_2\in E(G)$, we denote $\mu_G(x_1x_2)$ (or simply
$\mu(x_1x_2)$) the {\em multiplicity} of the edge $x_1x_2$ in $G$ 
(with $\mu(x_1x_2)=0$ if $x_1,x_2$ are nonadjacent), and we use 
$E_S(G)$ ($E_M(G)$) to denote the set of all simple nonpendant (multiple) 
edges of $G$, respectively.
If $x_1x_2\in E_M(G)$ and we need to distinguish the individual edges joining 
$x_1$ and $x_2$, we use the notation $(x_1x_2)^1$, $(x_1x_2)^2$ etc.
We say that vertices $x_1,x_2\in V(G)$ are {\em twins} in a multigraph $G$ if 
$\mu_G(x_1u)=\mu_G(x_2u)$ for each $u\in V(G)$.

We say that $S$ is an {\em induced subgraph} of a graph $G$
if $S$ can be obtained from $G$ by removing some vertices,
we denote this by $S \indsub G$.
We say that a (multi)graph $S$ is a {\em sub(multi)graph} of a multigraph $G$, 
denoted $S \subset G$, if $V(S)\subset V(G)$ and for every $x_1,x_2\in V(S)$ 
we have $\mu_S(x_1x_2)\leq\mu_G(x_1x_2)$, and we say that $S$ is a 
{\em flat subgraph} of $G$ if $S$ is a sub(multi)graph of $G$ where for every 
$x_1,x_2\in V(S)$, if $x_1x_2\in E(G)$, then $\mu_S(x_1x_2)\in\{1,\mu_G(x_1x_2)\}$.
We write $G_1\iso G_2$ if the (multi)graphs $G_1$, $G_2$ are isomorphic, 
and $\la M \rag$ to denote the {\em induced sub(multi)graph} on a set 
$M\subset V(G)$.

We use $d_G(x)$ to denote the {\em degree} of a vertex  $x$ in $G$ (note that 
if $G$ is a multigraph, then $d_G(x)$ equals the sum of multiplicities of the
edges containing $x$). 
For $x\in V(G)$, $N_G(x)$ denotes the {\em neighborhood} of $x$ in $G$, 
and for $M\subset V(G)$ we set $N_M(x)=N_G(x)\cap M$.

If $x\in V(G)$ is of degree 2 with $N_G(x)=\{y_1,y_2\}$, then the operation of 
replacing the path $y_1xy_2$ by the edge $y_1y_2$ is called {\em suppressing} the 
vertex $x$.
The inverse operation is called {\em subdividing} the edge $y_1y_2$ with the 
vertex $x$.
For $x,y\in V(G)$, $\dist_G(x,y)$ denotes the {\em distance} of $x$ and $y$ in 
$G$, and if $F\subset G$ is a connected subgraph and $x,y\in V(F)$, then 
$\dist_F(x,y)$ denotes the distance of $x,y$ in $F$, i.e., the length of a 
shortest $(x,y)$-path in $F$.
If $C$ is a cycle in $G$, then an edge $xy\in E(G)$ such that $x,y\in V(C)$ and 
$\dist_C(x,y)\geq 2$ is called a {\em chord} of $C$. If $C$ has no chords,
we say that $C$ is {\em chordless} (note that some edges of a chordless cycle 
can still be multiple in $H$). 

A triangle having a multiple edge is called a {\em multitriangle}
(see Fig.~\ref{fig-diamonds}$(a)$), and by a {\em diamond} we mean the graph
$K_4-e$ (see Fig.~\ref{fig-diamonds}$(b)$).
By a {\em clique} in $G$ we mean a complete subgraph of $G$, not necessarily 
maximal.

We say that a vertex $x\in V(G)$ is {\em simplicial} if $\la N_G(x)\rag$ is a 
clique, and we use $V_{SI}(G)$ to denote the set of all simplicial 
vertices of $G$, and $V_{NS}(G)=V(G)\setminus V_{SI}(G)$ the set of 
nonsimplicial vertices of $G$.
For $k\geq 1$, we say that a vertex $x\in V(G)$ is {\em locally $k$-connected 
in $G$} if $\la N_G(x)\rag$ is a $k$-connected graph.

A graph is {\em Hamilton-connected} if, for any $u,v\in V(G)$, $G$ has a 
hamiltonian $(u,v)$-path, i.e., an $(u,v)$-path $P$ with $V(P)=V(G)$.

Finally, if $\cF$ is a family of graphs, we say that $G$ is {\em $\cF$-free}
if $G$ does not contain an induced subgraph isomorphic to a member of $\cF$,
and the graphs in $\cF$ are referred to in this context as {\em forbidden
(induced) subgraphs}.
If $\cF=\{F\}$, we simply say that $G$ is {\em $F$-free}.
Here, the {\em claw} is the graph $\claw$, $P_i$ denotes the path on $i$ 
vertices, and $\Gamma_i$ denotes the graph obtained by joining two triangles 
with a path of length $i$ (see Fig.~\ref{fig-special_graphs}$(d)$).
Several further graphs that will occur as forbidden subgraphs are 
shown in Fig.~\ref{fig-special_graphs}$(a),(b),(c)$.
Whenever we will list vertices of an induced claw $\claw$, we will always list 
its center as the first vertex of the list, and when listing vertices of an 
induced subgraph $\Gamma_i$, we always list first the vertices of degree 2 
of one of the triangles, then the vertices of the path, and we finish with 
the vertices of degree 2 of the second triangle (i.e., in the labeling of
vertices as in Fig.~\ref{fig-special_graphs}$(d)$, we write 
$\lab t_1,t_2,p_1,\ldots,p_{i+1},t_3,t_4\rabg\iso\Gi$).

%
%
%
\begin{figure}[ht]
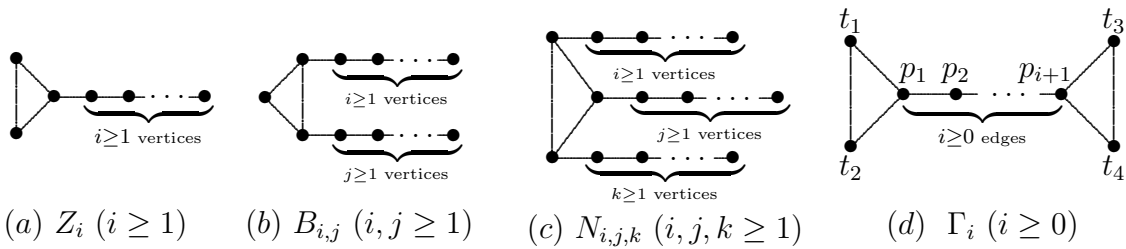

%
%
$$\bp
\setcoordinatesystem units <0.75mm,1.0mm>
\setplotarea x from -50 to 50, y from -5 to 5
\put
 {\bp
 \setcoordinatesystem units <1.0mm,1.0mm>
 \setplotarea x from 0 to 25, y from -5 to 5
 \put{$\bullet$} at    0    5
 \put{$\bullet$} at    0   -5
 \put{$\bullet$} at    5    0
 \put{$\bullet$} at   10    0
 \put{$\bullet$} at   15    0
 \put{$\ldots$}  at   20    0
 \put{$\bullet$} at   25    0
 \plot 5 0  0 -5  0 5  5 0  17 0 /
 \plot 23 0  25 0 /
 \put{$\underbrace{\hspace*{1.7cm}}_{i\geq 1\mbox{\tiny\ vertices}}$}
                             at 17.5 -4
 \put{$(a) \ Z_i$ $(i\geq 1$)} at  10.5 -17
 \ep}
at  -85   -1
\put
 {\bp
 \setcoordinatesystem units <1.0mm,1.0mm>
 \setplotarea x from -5 to 25, y from -5 to 5
 \put{$\bullet$} at    0    0
 \put{$\bullet$} at    5    5
 \put{$\bullet$} at   10    5
 \put{$\bullet$} at   15    5
 \put{$\ldots$}  at   20    5
 \put{$\bullet$} at   25    5
 \put{$\bullet$} at    5   -5
 \put{$\bullet$} at   10   -5
 \put{$\bullet$} at   15   -5
 \put{$\ldots$}  at   20   -5
 \put{$\bullet$} at   25   -5
 \plot 17 5 5 5 0 0 5 -5 17 -5 /
 \plot 23  5  25  5 /
 \plot 23 -5  25 -5 /
 \plot 5 -5  5 5 /
 \put{$\underbrace{\hspace*{1.7cm}}_{^{i\geq 1\mbox{\tiny\ vertices}}}$}
                             at 17.5  1.0
 \put{$\underbrace{\hspace*{1.7cm}}_{^{j\geq 1\mbox{\tiny\ vertices}}}$}
                             at 17.5  -9.0
 \put{$(b) \ B_{i,j}$ ($i,j\geq 1$)} at  12.5 -17
 \ep}
at   -42   -1.2
\put
 {\bp
 \setcoordinatesystem units <1.2mm,0.8mm>
 \setplotarea x from -25 to 25, y from -5 to 5
 \put{$\bullet$} at    0    5
 \put{$\bullet$} at    5    5
 \put{$\bullet$} at   10    5
 \put{$\ldots$}  at   15    5
 \put{$\bullet$} at   20    5
 \put{$\bullet$} at    5   -5
 \put{$\bullet$} at   10   -5
 \put{$\bullet$} at   15   -5
 \put{$\ldots$}  at   20   -5
 \put{$\bullet$} at   25   -5
 \put{$\bullet$} at    0  -15
 \put{$\bullet$} at    5  -15
 \put{$\bullet$} at   10  -15
 \put{$\ldots$}  at   15  -15
 \put{$\bullet$} at   20  -15
 \plot 0 5 5 -5 0 -15 0 5 /
 \plot 0  5  12 5 /
 \plot  18 5  20 5 /
 \plot 5 -5  17 -5 /
 \plot  25 -5  23 -5 /
 \plot 0  -15  12 -15 /
 \plot  18 -15  20 -15 /
 \put{$\underbrace{\hspace*{2.1cm}}_{^{i\geq 1\mbox{\tiny\ vertices}}}$}
                             at 12.5  0.4
 \put{$\underbrace{\hspace*{2.1cm}}_{^{j\geq 1\mbox{\tiny\ vertices}}}$}
                             at 17.5  -9.6
 \put{$\underbrace{\hspace*{2.1cm}}_{^{k\geq 1\mbox{\tiny\ vertices}}}$}
                             at 12.5  -19.6
 \put{$(c) \ N_{i,j,k}$ ($i,j,k\geq 1$)} at  12.5 -27
 \ep}
at   -4  0
\put
 {\bp
 \setcoordinatesystem units <1.4mm,1.4mm>
 \setplotarea x from 0 to 25, y from -5 to 5
 \put{$\bullet$} at    0   5
 \put{$\bullet$} at    0  -5
 \put{$\bullet$} at    5   0

 \put{$\bullet$} at   10   0
 \put{$\ldots$}  at   15   0
 \put{$\bullet$} at   20   0
 \put{$\bullet$} at   25   5
 \put{$\bullet$} at   25  -5
 \plot 5 0  0 -5  0 5  5 0  12 0 /
 \plot 18 0  20 0  25 5  25 -5  20 0 /
 \put{$t_1$}   at   0   7
 \put{$t_2$}   at   0  -7
 \put{$p_1$}   at   6  2
 \put{$p_2$}   at  10  2
 \put{$p_{i+1}$}  at  18.5  2
 \put{$t_3$}   at  25   7
 \put{$t_4$}   at  25  -7
 
 \put{$\underbrace{\hspace*{2.1cm}}_{i\geq 0\mbox{\tiny\ edges}}$}
                             at 12.5 -3.1
 \put{$(d) \ \ \Gamma_i$ $(i\geq 0)$} at  12.5 -12.5
 \ep}
at  70 1.8

\ep$$
\vspace*{-5mm}
\caption{The graphs $Z_i$, $B_{i,j}$, $N_{i,j,k}$ and $\Gi$}
\label{fig-special_graphs}
\end{figure}

\section{Introduction and main result}
\label{sec-main}

There are many results on forbidden induced subgraphs implying various 
Hamilton-type properties. While forbidden pairs of connected graphs 
for hamiltonicity in 2-connected graphs were completely characterized already 
in the early 90's \cite{Be91,FG97}, the progress in forbidden pairs for 
Hamilton-connectedness is relatively slow. 
For forbidden pairs of connected graphs, a list of potential candidates 
is known: one of them has to be the claw $\claw$, and the second one 
belongs to a list that will be mentioned in Section~\ref{sec-concluding}. 

\ms

Let $W$ denote the Wagner graph and $W^+$ the graph obtained from $W$ by
attaching exactly one pendant edge to each of its vertices 
(see Fig.~\ref{fig-Wagner}).

%
%
\begin{figure}[ht]
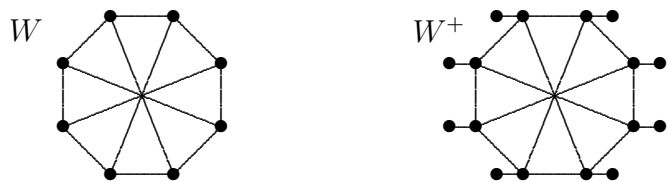

$$\beginpicture
\setcoordinatesystem units <0.7mm,1mm>
\setplotarea x from -70 to 70, y from -5 to 5
\put {\bp
\setcoordinatesystem units <0.35mm,0.35mm>
  \put{$\bullet$} at   -30   12
  \put{$\bullet$} at   -30  -12
  \put{$\bullet$} at    30   12
  \put{$\bullet$} at    30  -12
  \put{$\bullet$} at   -12   30
  \put{$\bullet$} at    12   30
  \put{$\bullet$} at   -12  -30
  \put{$\bullet$} at    12  -30
  \plot -30 -12 -30 12 -12 30 12 30 30 12 30 -12 12 -30
        -12 -30 -30 -12 /
  \plot -30  12  30 -12 /
  \plot -30 -12  30  12 /
  \plot -12  30  12 -30 /
  \plot -12 -30  12  30 /
 \put{$W$}   at  -44    25
\ep} at -40 0
\put {\bp
\setcoordinatesystem units <0.35mm,0.35mm>
  \put{$\bullet$} at   -30   12
  \put{$\bullet$} at   -30  -12
  \put{$\bullet$} at    30   12
  \put{$\bullet$} at    30  -12
  \put{$\bullet$} at   -12   30
  \put{$\bullet$} at    12   30
  \put{$\bullet$} at   -12  -30
  \put{$\bullet$} at    12  -30
  \plot -30 -12 -30 12 -12 30 12 30 30 12 30 -12 12 -30
        -12 -30 -30 -12 /
  \plot -30  12  30 -12 /
  \plot -30 -12  30  12 /
  \plot -12  30  12 -30 /
  \plot -12 -30  12  30 /
  \put{$\bullet$} at   -40   12
  \put{$\bullet$} at   -40  -12
  \put{$\bullet$} at    40   12
  \put{$\bullet$} at    40  -12
  \put{$\bullet$} at   -22   30
  \put{$\bullet$} at    22   30
  \put{$\bullet$} at   -22  -30
  \put{$\bullet$} at    22  -30
  \plot -40  12 -30  12 /
  \plot -40 -12 -30 -12 /
  \plot  40  12  30  12 /
  \plot  40 -12  30 -12 /
  \plot -22  30 -12  30 /
  \plot -22 -30 -12 -30 /
  \plot  22  30  12  30 /
  \plot  22 -30  12 -30 /
 \put{$W^+$}   at  -44    25
%
\ep} at 40 0
\endpicture$$
\vspace*{-6mm}
\caption{The Wagner graph $W$ and the graph $W^+$}
\label{fig-Wagner}
\vspace*{-2mm}
\end{figure}

Theorem~\ref{thmA-known_results} below lists the best known results on pairs of 
forbidden subgraphs implying Hamilton-connectedness of a 3-connected graph.

%
%
\begin{theoremAcite}{\cite{BGHJFW14,BFHTV02,LRVXY23-I,LRVXY23-II,LXL21,RV21,RV23}}
\label{thmA-known_results}
%
Let $G$ be a 3-connected $\{\claw,X\}$-free graph, where
\begin{mathitem}
\item {\bf ~\cite{BFHTV02}} $X=\Gamma_1$, or
\item {\bf ~\cite{BGHJFW14}} $X=P_9$, or
\item {\bf ~\cite{RV21}} $X=Z_7$ and $G\niso L(W^+)$, or 
\item {\bf ~\cite{RV23}} $X=B_{i,j}$ for $i+j\leq 7$, or
\item {\bf ~\cite{LRVXY23-I,LRVXY23-II,LXL21}} $X=N_{i,j,k}$ for $i+j+k\leq 7$.
\end{mathitem} 
Then $G$ is Hamilton-connected.
\end{theoremAcite}

Let $\cW$ be the family of graphs obtained by attaching at least one pendant edge
to each of the vertices of the Wagner graph $W$, and let 
$\cG=\{L(H)|\ H\in\cW\}$ be the family of their line graphs.
Then any $G\in\cG$ is 3-connected, non-Hamilton-connected, $P_{10}$-free,
$Z_8$-free, $B_{i,j}$-free for $i+j=8$ and $N_{i,j,k}$-free for $i+j+k=8$. 
Thus, this example shows that parts $(ii)$, $(iii)$, $(iv)$ and $(v)$ of 
Theorem~\ref{thmA-known_results} are sharp.

\bs

The following theorem is our main result.

%
%
\begin{theorem}
\label{thm-main}
Every $3$-connected $\{\claw,\Gamma_3\}$-free graph is Hamilton-connected.
\end{theorem}

\noi
{\bf Proof} of Theorem~\ref{thm-main} is postponed to 
Section~\ref{sec-proof-main}.

\bs

In Section~\ref{sec-preliminaries}, we collect necessary known results and 
facts on line graphs and on closure operations, and then, in 
Section~\ref{sec-Gt-closure}, we present a closure technique, introduced in 
the previous paper \cite{KRSV???-I}, that will be
crucial for the proof of the main result. 
Section~\ref{sec-proof-main} contains the proof of the main result,
in which the most technical parts (namely, the proof of Lemma~\ref{lemmaX},
and the introductory part and Case~1 of the proof of Theorem~\ref{thm-main}) 
are computer-assisted.
More details on the computation can be found in Section~\ref{sec-concluding}, 
and detailed results of the computation and source codes are available 
at~\cite{computing1} and \cite{computing2}.
Finally, in Section~\ref{sec-concluding}, we briefly update the discussion of 
remaining open cases in the characterization of forbidden pairs for 
Hamilton-connectedness from \cite{LRVXY23-II} and~\cite{RV23}.

\section{Preliminaries}
\label{sec-preliminaries}

In Subsections~\ref{subsec-linegraph-preimage} -- \ref{subsec-closure}, 
we summarize some known facts that will be needed in the proof of 
Theorem~\ref{thm-main}.

\subsection{Line graphs of multigraphs and their preimages}
\label{subsec-linegraph-preimage}

The following characterization of line graphs of multigraphs was proved 
by Bermond and Meyer~\cite{BM73} (see also Zverovich~\cite{Z97}).

%
%
\begin{theoremAcite}{\cite{BM73}}
\label{thmA-BeMe}
A graph $G$ is a line graph {\em of a multigraph} if and only if
$G$ does not contain a copy of any of the graphs in Figure~\ref{fig-BeMe} as an induced subgraph.
\end{theoremAcite}

\vspace*{-5mm}

%
%
\begin{figure}[ht]
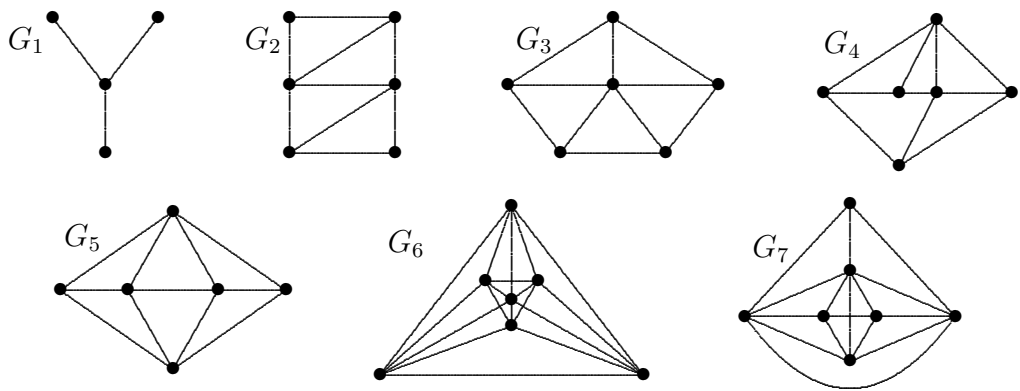

$$\bp
\setcoordinatesystem units <0.9mm,0.75mm>
\setplotarea x from -50 to 50, y from 0 to 45
\put{
\bp
\setcoordinatesystem units <0.7mm,.6mm>
\setplotarea x from -20 to 20, y from -20 to 15
\put{$\bullet$} at -10  15
\put{$\bullet$} at  10  15
\put{$\bullet$} at   0   0
\put{$\bullet$} at   0 -15
\plot -10 15  0 0  10 15 /
\plot 0 0  0 -15 /
\put{$G_1$} at -15 10
\ep} at -60 35
\put{
\bp
\setcoordinatesystem units <0.7mm,.6mm>
\setplotarea x from -20 to 20, y from -20 to 15
\put{$\bullet$} at -10  15
\put{$\bullet$} at  10  15
\put{$\bullet$} at -10   0
\put{$\bullet$} at  10   0
\put{$\bullet$} at -10 -15
\put{$\bullet$} at  10 -15
\plot -10 -15 -10 15 10 15 10 -15 -10 -15 10 0 -10 0 10 15 /
\put{$G_2$} at -15 10
\ep} at -25 35
\put{
\bp
\setcoordinatesystem units <0.7mm,.6mm>
\setplotarea x from -20 to 20, y from -20 to 15
\put{$\bullet$} at   0  15
\put{$\bullet$} at   0   0
\put{$\bullet$} at -20   0
\put{$\bullet$} at  20   0
\put{$\bullet$} at -10 -15
\put{$\bullet$} at  10 -15
\plot 0 15 -20 0 -10 -15 10 -15 20 0 0 15 0 0 -10 -15 /
\plot -20 0 20 0 /
\plot 0 0  10 -15 /
\put{$G_3$} at -15 10
\ep} at 15 35
\put{
\bp
\setcoordinatesystem units <0.5mm,.65mm>
\setplotarea x from -20 to 20, y from -20 to 20
\put{$\bullet$} at -25   0
\put{$\bullet$} at  -5   0
\put{$\bullet$} at   5   0
\put{$\bullet$} at  25   0
\put{$\bullet$} at   5  15
\put{$\bullet$} at  -5 -15
\plot -25 0 25 0  5 15 -25 0  -5 -15 5 0  5 15  -5 0 /
\plot -5 -15  25 0 /
\put{$G_4$} at -20 10
\ep} at  60 35
\put{
\bp
\setcoordinatesystem units <0.6mm,.7mm>
\setplotarea x from -20 to 20, y from -20 to 20
\put{$\bullet$} at   0  15
\put{$\bullet$} at   0 -15
\put{$\bullet$} at -25   0
\put{$\bullet$} at -10   0
\put{$\bullet$} at  10   0
\put{$\bullet$} at  25   0
\plot -25 0  25 0  0 15 -25 0  0 -15 -10 0  0 15  10 0  0 -15 25 0 /
\put{$G_5$} at -20 10
\ep} at -50 0
\put{
\bp
\setcoordinatesystem units <0.7mm,.5mm>
\setplotarea x from -20 to 20, y from -20 to 20
\put{$\bullet$} at   5  20
\put{$\bullet$} at   0   0
\put{$\bullet$} at  10   0
\put{$\bullet$} at   5  -5
\put{$\bullet$} at   5 -12
\put{$\bullet$} at -20 -25
\put{$\bullet$} at  30 -25
\plot 5 20 -20 -25  30 -25  5 20  0 0  -20 -25  5 -12  30 -25  10 0
    5 20  5 -12  0 0  10 0  5 -12 /
\plot -20 -25  5 -5  0 0 /
\plot 30 -25  5 -5  10 0 /
\put{$G_6$} at -15 10
\ep} at   0 0
\put{
\bp
\setcoordinatesystem units <0.7mm,.6mm>
\setplotarea x from -20 to 20, y from -20 to 20
\put{$\bullet$} at   0  20
\put{$\bullet$} at   0   5
\put{$\bullet$} at -20  -5
\put{$\bullet$} at  20  -5
\put{$\bullet$} at  -5  -5
\put{$\bullet$} at   5  -5
\put{$\bullet$} at  0  -15
\plot 0 20 -20 -5  20 -5  0 20  0 -15 -20 -5  0 5  20 -5
   0 -15  -5 -5  0 5  5 -5  0 -15 /
\setquadratic
\plot -20 -5  0 -21  20 -5 /
\setlinear
\put{$G_7$} at -15 10
\ep} at   50 0
\ep$$
\caption{Forbidden subgraphs for line graphs of multigraphs}
\label{fig-BeMe}
\end{figure}

While in line graphs of graphs, for a connected line graph $G$, the graph $H$ 
such that $G=L(H)$ is uniquely determined with a single exception of $G=K_3$, 
in line graphs of multigraphs this is not true: a simple example is the graphs 
$H_1=Z_1$ and $H_2$ a double edge with one pendant edge attached to each vertex
-- while $H_1\not\iso H_2$, we have $L(H_1)\iso L(H_2)$.
Using a modification of an approach from \cite{Z97}, the following was proved 
in \cite{RV11}.

%
%
\begin{theoremAcite}{\cite{RV11}}
\label{thmA-vzor_jedn}
Let $G$ be a connected line graph of a multigraph. Then there is, up to 
an isomorphism, a uniquely determined multigraph $H$ such that $G=L(H)$ and
a vertex $e\in V(G)$ is simplicial in $G$ if and only if the corresponding 
edge $e\in E(H)$ is a pendant edge in $H$.
\end{theoremAcite}

The multigraph $H$ with the properties given in Theorem~\ref{thmA-vzor_jedn}
will be called the {\em preimage} of a line graph $G$ and denoted $H=\Lp(G)$.
We will also use the notation $a=L(e)$ and $e=\Lp(a)$ for an edge $e\in E(H)$
and the corresponding vertex $a\in V(G)$.

\ms

An edge-cut $R\subset E(H)$ of a multigraph $H$ is {\em essential} if $H-R$ has 
at least two nontrivial components, and $H$ is 
{\em essentially $k$-edge-connected} if every essential edge-cut of $H$ is 
of size at least $k$. It is a well-known fact that a line graph $G$ is 
$k$-connected if and only if $\Lp(G)$ is essentially $k$-edge-connected.
It is also a well-known fact that if $X$ is a line graph, then a line graph $G$
is $X$-free if and only if $\Lp(G)$ does not contain as a subgraph (not 
necessarily induced) a graph $F$ such that $L(F)=X$.

Note that in the special case of the graph $\Gt$, there are three 
nonisomorphic multigraphs $F_1,F_2,F_3$ such that $L(F_i)=\Gt$, 
$i=1,2,3$, see Fig.~\ref{fig-Preimages}.
It is straightforward to verify that if $G=\Gt$, then $\Lp(G)=F_1$, 
however, if $G$ contains $F\iso\Gt$ as a proper induced subgraph, then the 
corresponding subgraph $\Lp(F)$ can be any of $F_1,F_2,F_3$ (for example, 
if $G$ is obtained from $\Gt$ by adding to $F=\Gt$ new vertices $u_1,u_2$
and edges $u_1u_2,u_2t_1,u_2t_2$, then in $\Lp(G)$, $\Lp(F)$ is the 
multigraph $F_2$).

%
%
\begin{figure}[ht]
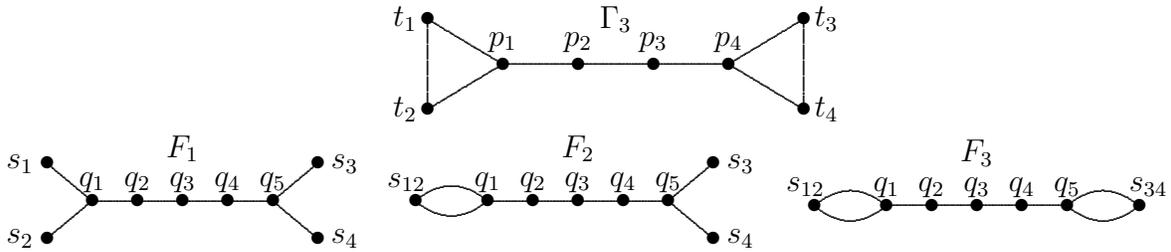

$$\beginpicture
\setcoordinatesystem units <1.2mm,1mm>
\setplotarea x from -70 to 70, y from -5 to 5
\put {\bp
\setcoordinatesystem units <1.0mm,1.2mm>
  \put{$\bullet$} at  -25  -5
  \put{$t_2$} at  -28  -5
  \put{$\bullet$} at  -25   5
  \put{$t_1$} at  -28   5
  \put{$\bullet$} at  -15   0
  \put{$p_1$} at  -15   2.5
  \put{$\bullet$} at   -5   0
  \put{$p_2$} at   -5   2.5
  \put{$\bullet$} at    5   0
  \put{$p_3$} at    5   2.5
  \put{$\bullet$} at   15   0
  \put{$p_4$} at   15   2.5
  \put{$\bullet$} at   25  -5
  \put{$t_4$} at   28  -5
  \put{$\bullet$} at   25   5
  \put{$t_3$} at   28   5
  \plot -15 0  -25 -5  -25 5 -15 0  15 0  25 5  25 -5  15 0 /
 \put{$\Gamma_3$}   at  0  5
\ep} at 0 17
\put {\bp
\setcoordinatesystem units <1.2mm,1.0mm>
  \put{$\bullet$} at  -15  -5
  \put{$\bullet$} at  -15   5
  \put{$\bullet$} at  -10   0
  \put{$\bullet$} at   -5   0
  \put{$\bullet$} at    0   0
  \put{$\bullet$} at    5   0
  \put{$\bullet$} at   10   0
  \put{$\bullet$} at   15  -5
  \put{$\bullet$} at   15   5
  \put{$s_1$} at  -18   5
  \put{$s_2$} at  -18  -5
  \put{$s_3$} at   18   5
  \put{$s_4$} at   18  -5
  \put{$q_1$} at  -10   2.5
  \put{$q_2$} at   -5   2.5
  \put{$q_3$} at    0   2.5
  \put{$q_4$} at    5   2.5
  \put{$q_5$} at   10   2.5
  \plot -10 0   10 0  /
  \plot -15 5  -10 0  -15 -5  /
  \plot  15 5   10 0   15 -5  /
 \put{$F_1$}   at  0  7
\ep} at -48 0
\put {\bp
\setcoordinatesystem units <1.2mm,1.0mm>
  \put{$\bullet$} at  -18   0
  \put{$\bullet$} at  -10   0
  \put{$\bullet$} at   -5   0
  \put{$\bullet$} at    0   0
  \put{$\bullet$} at    5   0
  \put{$\bullet$} at   10   0
  \put{$\bullet$} at   15  -5
  \put{$\bullet$} at   15   5
  \put{$s_{12}$} at  -19   2.5
  \put{$s_3$} at   18   5
  \put{$s_4$} at   18  -5
  \put{$q_1$} at  -10   2.5
  \put{$q_2$} at   -5   2.5
  \put{$q_3$} at    0   2.5
  \put{$q_4$} at    5   2.5
  \put{$q_5$} at   10   2.5
  \plot -10 0   10 0  /
  \plot  15 5   10 0   15 -5  /
  \setquadratic
  \plot -18 0  -14  2  -10 0 /
  \plot -18 0  -14 -2  -10 0 /
  \setlinear
 \put{$F_2$}   at  0  7
\ep} at  -5 0
\put {\bp
\setcoordinatesystem units <1.2mm,1.0mm>
  \put{$\bullet$} at  -18   0
  \put{$\bullet$} at   18   0
  \put{$\bullet$} at  -10   0
  \put{$\bullet$} at   -5   0
  \put{$\bullet$} at    0   0
  \put{$\bullet$} at    5   0
  \put{$\bullet$} at   10   0
%
  \put{$s_{12}$} at  -19   2.5
  \put{$s_{34}$} at   19   2.5
  \put{$q_1$} at  -10   2.5
  \put{$q_2$} at   -5   2.5
  \put{$q_3$} at    0   2.5
  \put{$q_4$} at    5   2.5
  \put{$q_5$} at   10   2.5
  \plot -10 0   10 0  /
  \setquadratic
  \plot -18 0  -14  2  -10 0 /
  \plot -18 0  -14 -2  -10 0 /
  \plot  18 0   14  2   10 0 /
  \plot  18 0   14 -2   10 0 /
  \setlinear
 \put{$F_3$}   at  0  7
\ep} at  40 2
\endpicture$$
\caption{The graph $\Gt$ and its three preimages}
\label{fig-Preimages}
\end{figure}

Recall that a closed trail $T$ is a {\em dominating closed trail} (abbreviated DCT) if 
$T$ dominates all edges of $G$, and an $(e,f)$-trail is {\em an internally dominating 
$(e,f)$-trail} (abbreviated $(e,f$)-IDT) if $\Int(T)$ dominates all edges of $G$.
Harary and Nash-Williams \cite{HNW65} established a correspondence between a DCT 
in $H$ and a hamiltonian cycle in $L(H)$.
A similar result showing that $G=L(H)$ is Hamilton-connected if and only 
if $H$ has an $(e_1,e_2)$-IDT for any pair of edges $e_1,e_2\in E(H)$, was 
given in \cite{LLZ05}
(in fact, part $(ii)$ of the following theorem is slightly stronger than the 
result from \cite{LLZ05}, and its easy proof is given in \cite{LRVXY23-I}).
Note that these results were proved for line graphs of graphs but it is easy 
to verify that they remain true also for line graphs of multigraphs.

%
%
\begin{theoremAcite}{\cite{HNW65,LLZ05}}
\label{thmA-DCT+IDT} 
Let $H$ be a multigraph with $|E(H)|\geq 3$ and let $G=L(H)$.
\begin{mathitem}
\item {\bf\cite{HNW65}} The graph $G$ is hamiltonian if and only if $H$ has a DCT.
\item {\bf\cite{LLZ05}}  For every $e_i\in E(H)$ and $a_i=L(e_i)$, $i=1,2$,
      $G$ has a hamiltonian $(a_1,a_2)$-path if and only if $H$ has an 
      $(e_1,e_2)$-IDT.
\end{mathitem}
\end{theoremAcite}

\bsm

\subsection{Strongly spanning trailable multigraphs}
\label{subsec-strongly_spanning_tr}

A multigraph $H$ is {\em strongly spanning trailable} if for any 
$e_1=u_1v_1,e_2=u_2v_2\in E(H)$ (possibly $e_1=e_2)$, the multigraph $H(e_1,e_2)$, 
which is obtained from $H$ by replacing the edge $e_1$ by a path $u_1v_{e_1}v_1$ 
and the edge $e_2$ by a path $u_2v_{e_2}v_2$, has a spanning 
$(v_{e_1},v_{e_2})$-trail. 

By Theorem~\ref{thmA-DCT+IDT}$(ii)$, it is straightforward to see that if
$H$ is strongly spanning trailable, then $G=L(H)$ is Hamilton-connected.

\ms

We will need the following two results on ``small" strongly spanning trailable 
multigraphs from \cite{LXL21}.
Here, $\dvW$ is the set of multigraphs that are obtained from the Wagner graph 
$W$ by subdividing one of its edges and adding at least one edge between the new 
vertex and exactly one of its neighbors.

%
%
\begin{theoremAcite}{\cite{LXL21}}
\label{thmA-strongly_trailable}
\begin{mathitem}
\item Every 2-connected 3-edge-connected multigraph $H$ with circumference 
   $c(H)\leq 8$ other than the Wagner graph $W$ is strongly spanning trailable.
\item Every 3-edge-connected multigraph $H$ with $|V(H)|\leq 9$ such that 
      $H\notin \{W\}\cup \dvW$ is strongly spanning trailable.
\end{mathitem}
\end{theoremAcite}

\subsection{The core of the preimage of a 3-connected line graph}
\label{subsec-core}

To avoid difficulties that can occur with the core of a multigraph, we define 
the core only for the case we need, i.e., for the preimage of a 3-connected 
line graph (then e.g. vertices of degree~2 are independent by the connectivity 
assumption, and pendant multiedges cannot occur by Theorem~\ref{thmA-vzor_jedn}).

Thus, let $G$ be a 3-connected line graph and let $H=\Lp(G)$. 
The {\em core of $H$} is the multigraph $\co(H)$ obtained from $H$ by 
removing all pendant edges and suppressing all vertices of degree 2.

Shao~\cite{S05} proved the following properties of the core of a multigraph.

%
%
\begin{theoremAcite}{\cite{S05}}
\label{thmA-core}
Let $H$ be an essentially 3-edge-connected multigraph. Then
\begin{mathitem}
\item $\co(H)$ is uniquely determined,
\item $\co(H)$ is 3-edge-connected,
\item $V(\co(H))$ dominates all edges of $H$,
\item if $\co(H)$ has a spanning closed trail, then $H$ has a DCT,
\item if $\co(H)$ is strongly spanning trailable, then $L(H)$ is Hamilton-connected.
\end{mathitem}
\end{theoremAcite}

\subsection{Closure operations}
\label{subsec-closure}

For $x\in V(G)$, the {\em local completion of $G$ at $x$} is the
graph $G^{^*}_x=(V(G),E(G)\cup\{y_1y_2|\ y_1,y_2\in N_G(x)\})$ (i.e.,
$G^{^*}_x$ is obtained from $G$ by adding all the missing edges with
both vertices in $N_G(x)$). In this context, the edges in $E(\Gstx)\sm E(G)$
will be refereed to {\em new edges}, and the edges in $E(G)$ are {\em old}.
Obviously, if $G$ is claw-free, then so is $G^{^*}_x$.
Note that in the special case when $G$ is a line graph and $H=\Lp(G)$, 
$G^{^*}_x$ is the line graph of the multigraph $H|_{e}$ obtained from $H$ by 
contracting the edge $e=\Lp(x)$ into a vertex and replacing the created 
loop(s) by pendant edge(s) (Thus, if $G=L(H)$ and $x=L(e)$, then 
$\Gst_x=L(H|_e)$).

Also note that clearly $x\in V_{SI}(G^{^*}_x)$ for any $x\in V(G)$, and, more 
generally, $V_{SI}(G)\subset V_{SI}(G^{^*}_x)$ for any $x\in V(G)$.

\ms

We say that a vertex $x\in V(G)$ is {\em eligible} if $\la N_G(x)\rag$ 
is a connected noncomplete graph, and we use $V_{EL}(G)$ to denote the set of 
all eligible vertices of $G$. Note that in the special case when $G$ is a 
line graph and $H=\Lp(G)$, it is not difficult to observe that $x\in V(G)$ is 
eligible if and only if the edge $\Lp(x)$ is in a triangle or in a multiple edge
of $H$.
Based on the fact that if $G$ is claw-free and $x\in V_{EL}(G)$, then
$G^{^*}_x$ is hamiltonian if and only if $G$ is hamiltonian, the
{\em closure} $\cl(G)$ of a claw-free graph $G$ was defined in \cite{R97} as
the graph obtained from $G$ by recursively performing the local completion
operation at eligible vertices, as long as this is possible (more precisely:
$\cl(G)=G_k$, where $G_1,\ld,G_k$ is a sequence of graphs such that $G_1=G$,
$G_{i+1}=(G_i)^{^*}_{x_i}$ for some $x_i\in V_{EL}(G)$, $i=1,\ld,k-1$,
and $V_{EL}(G_k)=\emptyset$).
The closure $\cl(G)$ of a claw-free graph $G$ is uniquely determined, is a 
line graph of a triangle-free graph, and is hamiltonian if and only if so 
is $G$. However, as observed in \cite{BFR00}, the closure operation
does not preserve the (non-)Hamilton-connectedness of~$G$.

\bs

To handle this problem, the closure concept was strengthened in 
\cite{LRVXY23-I} by omitting the eligibility assumption for the application 
of the local completion operation.
Specifically, for a given claw-free graph $G$, we construct a graph $G^U$ 
by the following construction.

\begin{mathitem}
\item If $G$ is Hamilton-connected, we set $G^U=K_{|V(G)|}$.
\item If $G$ is not Hamilton-connected, we recursively perform the
  local completion operation at such vertices for which the
  resulting graph is still not Hamilton-connected, as long as this is
  possible. We obtain a sequence of graphs $G_1,\ld,G_k$ such that
  \begin{mathitem}
    \item[$\bullet$] $G_1=G$,
    \item[$\bullet$] $G_{i+1}=(G_i)^{^*}_{x_i}$ for some
        $x_i\in V(G_i)$, $i=1,\ld,k-1$,
    \item[$\bullet$] $G_k$ has no hamiltonian $(a,b)$-path for some
        $a,b\in V(G_k)$,
     \item[$\bullet$] for any $x\in V(G_k)$, $(G_k)^{^*}_x$ is
        Hamilton-connected,
  \end{mathitem}
  and we set $G^U=G_k$.
\end{mathitem}
A graph $G^U$ obtained by the above construction is called an
{\em ultimate M-closure} (or briefly a {\em UM-closure}) of the
graph $G$, and a graph $G$ equal to its UM-closure is said to be 
{\em UM-closed}.

\ms

The following theorem summarizes basic properties of the UM-closure operation.

%
%
\begin{theoremAcite}{\cite{LRVXY23-I}} 
\label{thmA-ultim-clos}
Let $G$ be a claw-free graph and let $G^U$ be one of its UM-closures.
Then $G^U$ has the following properties:
\begin{mathitem}
\item $V(G)=V(G^U)$ and $E(G)\subset E(G^U)$,
\item $G^U$ is obtained from $G$ by a sequence of local completions at
        vertices,
\item $G$ is Hamilton-connected if and only if $G^U$ is
         Hamilton-connected,
\item if $G$ is Hamilton-connected, then $G^U=K_{|V(G)|}$,
\item if $G$ is not Hamilton-connected, then $(G^U)^{^*}_x$ is
              Hamilton-connected for any $x\in V(G^U)$,
\item $G^U=L(H)$, where $\co(H)$ contains no diamond, no multitriangle and no 
       triple edge, and either
   \begin{mathitem}
   \item[$(\alpha)$] at most 2 triangles and no multiedge, or
   \item[$(\beta)$] no triangle, at most one double edge and no other
         multiedge, and if $\co(H)$ contains a double edge, then this double edge 
         is also in $H$,
   \end{mathitem}
\item[$(vii)$] if $G^U$ contains no hamiltonian $(a,b)$-path for some
   $a,b\in V(G^U)$ and
   \begin{mathitem}
   \item[$(\alpha)$] $X$ is a triangle in $\co(H)$, then
         $E(X)\cap\{L_{G^U}^{-1}(a),L_{G^U}^{-1}(b)\}\neq\emptyset$,
   \item[$(\beta)$] $X$ is a multiedge in $\co(H)$, then
         $E(X)=\{L_{G^U}^{-1}(a),L_{G^U}^{-1}(b)\}$.
   \end{mathitem}
\end{mathitem}
\end{theoremAcite}

We will also need the following lemma from \cite{RV14}.

%
%
\begin{lemmaAcite}{\cite{RV14}} 
\label{lemmaA-SMclos-degree2}
Let $G$ be an SM-closed graph and let $H=\Lp(G)$. Then $H$ does not contain 
a triangle with a vertex of degree 2 in $H$.
\end{lemmaAcite}

Note that Lemma~\ref{lemmaA-SMclos-degree2} was proved in \cite{RV14} for 
SM-closed graphs (which we do not define here), but since every UM-closed graph 
is also SM-closed (see e.g. \cite{LRVXY23-I}), it is true also for UM-closed 
graphs.

\section{$\Gt$-closure}
\label{sec-Gt-closure}

The UM-closure operation preserves Hamilton-connectedness, but there is
still a problem that the local completion $G^{^*}_{x}$ of a $\{\claw,\Gt\}$-free 
graph $G$ is not necessarily $\Gt$-free. 
To handle this problem, we define the concept of a $\Gt$-closure $G^{\Gt}$ of a 
$\{\claw,\Gt\}$-free graph $G$.
For a set $M=\{x_1,x_2,\ldots,x_k\}\subset V(G)$, we set 
$G^{^*}_M=((G^{^*}_{x_1})^{^*}_{x_2}\ldots )^{^*}_{x_k}$. It is 
implicit in the proof of uniqueness of $\cl(G)$ in \cite{R97} (and easy
to see) that, for a given set $M=\{x_1,x_2,\ldots,x_k\}\subset V(G)$,
$G^{^*}_M$ is uniquely determined (i.e., does not depend on the order of the 
vertices $x_1,x_2,\ldots,x_k$ used during the construction).

\ms

If $G$ is not Hamilton-connected, then a vertex $x\in V_{NS}(G)$, for 
which the graph $\Gstx$ is still not Hamilton-connected, is said to be 
{\em feasible in $G$}. A set of vertices $M\subset V(G)$ is said to be 
{\em feasible in $G$} if the vertices in $M$ can be ordered in a sequence 
$x_1,\ldots,x_k$ such that $x_1$ is feasible in $G_0=G$, and $x_{i+1}$
is feasible in $G_i=(G_{i-1})^{^*}_{x_i}$, $i=1,\ldots,k-1$.
Thus, if $M\subset V(G)$ is feasible, then $M\subset V_{SI}(\Gst_M)$, 
but $\Gst_M$ is still not Hamilton-connected.

Note that it is possible that some two vertices $x,y$ of a graph $G$ 
are feasible in $G$, but $x$ is not feasible in $\Gst_y$ 
(for example, if $H$ is obtained from the Petersen graph by adding a
pendant edge to each vertex, subdividing a nonpendant edge $x_1x_2$ 
with a vertex $w$, replacing each of the edges $x_iw$ with a double edge, 
and if $G=L(H)$ and $x'_i,x''_i\in V(G)$ correspond to the two edges joining 
$x_i$ and $w$ in $H$, $i=1,2$, then $G$ is not Hamilton-connected, each of
the vertices $x'_i,x''_i$ is feasible in $G$, $i=1,2$, but e.g. $x'_1$
and $x''_1$ are not feasible in $\Gst_{x'_2}\iso\Gst_{x''_2}$).
Thus, the recursive form of the definition is essential for verifying 
feasibility of a set $M\subset V(G)$ (although the resulting graph $\Gst_M$ 
does not depend on their order).

Recall that in the special case when $G=L(H)$, a local completion at a vertex 
$x\in V(G)$ corresponds to the contraction of the corresponding edge 
$e=\Lp(H)$. In this case, when $x=L(e)$ is feasible in $G=L(H)$, we also say 
that the edge $e\in E(H)$ is {\em contractible in $H$}, and, similarly, 
if a set $M\subset V(G)$ is feasible, then the corresponding set of edges
$\Lp(M)\subset E(H)$ is said to be contractible in $H$.

\ms

Now, for a $\{\claw,\Gt\}$-free graph $G$, we define its $\Gt$-closure 
$G^{\Gt}$ by the following construction.
\begin{mathitem}
\item If $G$ is Hamilton-connected, we define $G^{\Gt}$ as the complete 
      graph.
\item If $G$ is not Hamilton-connected, we recursively perform the
  local completion operation at such feasible sets of vertices for which 
  the resulting graph is still $\Gt$-free, as long as this is possible. 
  We obtain a sequence of graphs $G_1,\ld,G_k$ such that
  \begin{mathitem}
    \item[$\bullet$] $G_1=G$,
    \item[$\bullet$] $G_{i+1}=(G_i)^{^*}_{M_i}$ for some
        set $M_i\subset V(G_i)$, $i=1,\ld,k-1$,
    \item[$\bullet$] $G_k$ has no hamiltonian $(a,b)$-path for some
        $a,b\in V(G_k)$,
    \item[$\bullet$] for any feasible set $M\subset V_{NS}(G_k)$, 
    $(G_k)^{^*}_M$ contains an induced subgraph isomorphic to $\Gt$,
  \end{mathitem}
  and we set $G^{\Gt}=G_k$.
\end{mathitem}
A resulting graph $G^{\Gt}$ is called a {\em $\Gt$-closure} of the graph 
$G$, and a graph $G$ equal to (some) its $\Gt$-closure is said to be 
{\em $\Gt$-closed}.
Note that for a given graph $G$, its $\Gt$-closure is not uniquely determined.

\ms

The following two theorems give basic properties of the $\Gt$-closure
operation, proved in \cite{KRSV???-I} (for the graphs $W_5$, $W_4$, $P_6^2$ 
and $P_6^{2+}$, see Fig.~\ref{fig-kola_a_cesty}).

%
%
\begin{figure}[ht]
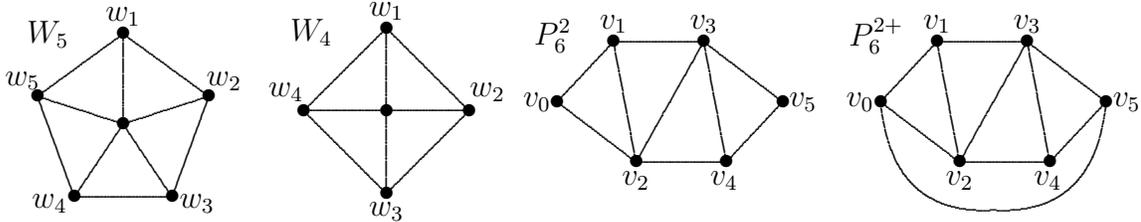

$$\bp
\setcoordinatesystem units <1mm,1mm>
\setplotarea x from -50 to 50, y from -7 to 7
\put{\beginpicture
\setcoordinatesystem units <.1mm,.1mm>
\setplotarea x from -120 to 120, y from -120 to 120
\put{$\bullet$} at     0   0
\put{$\bullet$} at     0 120
\put{$\bullet$} at   114  37
\put{$\bullet$} at    65 -97
\put{$\bullet$} at   -65 -97
\put{$\bullet$} at  -114 37
\put{$w_1$} at     0  143
\put{$w_2$} at   134   57
\put{$w_3$} at    98 -105
\put{$w_4$} at   -98 -105
\put{$w_5$} at  -134   58
\plot  0 120   114 37  65 -97  -65 -97  -114 37   0 120 /
\plot  0 0    0 120 /
\plot  0 0  114  37 /
\plot  0 0   65 -97 /
\plot  0 0  -65 -97 /
\plot  0 0 -114  37 /
\put{$W_5$} at  -100 120
\endpicture} at -65  0
\put{\beginpicture
\setcoordinatesystem units <.1mm,.1mm>
\setplotarea x from -120 to 120, y from -120 to 120
\put{$\bullet$} at     0    0
\put{$\bullet$} at     0  110
\put{$\bullet$} at   110    0
\put{$\bullet$} at     0 -110
\put{$\bullet$} at  -110    0
\put{$w_1$} at     0  133
\put{$w_2$} at   134   25
\put{$w_3$} at     0 -133
\put{$w_4$} at  -134   25
\plot  0 110   110 0  0 -110  -110 0  0 110 /
\plot  0 0    0  110 /
\plot  0 0  110    0 /
\plot  0 0    0 -110 /
\plot  0 0 -110    0 /
\put{$W_4$} at  -100 105
\endpicture} at -30  0
\put{
\bp
\setcoordinatesystem units <0.3mm,.8mm>
\setplotarea x from -20 to 20, y from -15 to 10
\put{$\bullet$} at  -50   0
\put{$\bullet$} at  -25  10
\put{$\bullet$} at  -15 -10
\put{$\bullet$} at   15  10
\put{$\bullet$} at   25 -10
\put{$\bullet$} at   50   0
\plot -50 0 -25 10  15 10  50 0  25 -10 -15 -10  -50 0 /
\plot -25 10   -15 -10  15 10  25 -10 /
\put{$v_0$} at  -59   0
\put{$v_1$} at  -25  13
\put{$v_2$} at  -15 -13
\put{$v_3$} at   15  13
\put{$v_4$} at   24 -13
\put{$v_5$} at   59   0
\put{$P_6^{2}$} at -52 11
\ep} at  7 1
\put{
\bp
\setcoordinatesystem units <0.3mm,.8mm>
\setplotarea x from -20 to 20, y from -15 to 10
\put{$\bullet$} at  -50   0
\put{$\bullet$} at  -25  10
\put{$\bullet$} at  -15 -10
\put{$\bullet$} at   15  10
\put{$\bullet$} at   25 -10
\put{$\bullet$} at   50   0
\plot -50 0 -25 10  15 10  50 0  25 -10 -15 -10  -50 0 /
\plot -25 10   -15 -10  15 10  25 -10 /
\put{$v_0$} at  -59   0
\put{$v_1$} at  -25  13
\put{$v_2$} at  -15 -13
\put{$v_3$} at   15  13
\put{$v_4$} at   24 -13
\put{$v_5$} at   59   0
\setquadratic
\plot -50 0  -35 -14   0 -18   35 -14   50 0 /
\setlinear
\put{$P_6^{2+}$} at -52 11
\ep} at   50 1
\ep$$
\bsm
\caption{The 5-wheel $W_5$, the 4-wheel $W_4$, and the graphs $P_6^2$ and 
        $P_6^{2+}$}
\label{fig-kola_a_cesty}
\end{figure}
%
%

%
%
\begin{theoremAcite}{\cite{KRSV???-I}}
\label{thmA-closure-W4-W5-P6}
Let $G$ be a $\{\claw,\Gt\}$-free graph and let $G^{\Gt}$ be its $\Gt$-closure.
Then $G^{\Gt}$ is $\{\claw,W_5,W_4,P_6^2,P_6^{2+}\}$-free.
\end{theoremAcite}

To show that $G^{\Gt}$ is a line graph of a multigraph, by 
Theorem~\ref{thmA-BeMe}, it is sufficient to show that $G^{\Gt}$  does not 
contain as an induced subgraph any of the graphs $G_1,\ldots,G_7$ of 
Figure~\ref{fig-BeMe}.
Since $G_1\iso\claw$, $G_3\iso W_5$, each of the graphs 
$G_5,G_6,G_7$ contains an induced $W_4$, $G_2\iso P_6^2$, and 
$G_4\iso P_6^{2+}$, Theorem~\ref{thmA-closure-W4-W5-P6} immediately 
implies the following crucial fact.

%
%
\begin{theoremAcite}{\cite{KRSV???-I}}
\label{thmA-Gamma3-clos-lajngraf}
Let $G$ be a $\{\claw,\Gt\}$-free graph and let $G^{\Gt}$ be its 
$\Gt$-closure. Then there is a multigraph $H$ such that $G^{\Gt}=L(H)$.
\end{theoremAcite}

Further structural properties of a $\Gt$-closure of a graph and of its preimage
will be shown in Section~\ref{sec-proof-main} 
(Lemma~\ref{lemma-diamonds} and Lemma~\ref{lemma-minimality}).

\bs

The following results will be useful to identify feasible vertices.

%
%
\begin{theoremAcite}{\cite{RV10}}
\label{thmA-HC-2-conn_neighb}
Let $G$ be a claw-free graph and let $x\in V(G)$ be locally 2-connected
in $G$. Then $G$ is Hamilton-connected if and only if $\Gstx$ is 
Hamilton-connected.
\end{theoremAcite}

Thus, in our terminology, Theorem~\ref{thmA-HC-2-conn_neighb} says that 
a locally 2-connected vertex is feasible.

%
%
\begin{lemmaAcite}{\cite{RV11}}
\label{lemmaA-2_indep_sets}
Let $G$ be a claw-free graph, $x\in V(G)$, and let
$H\indsub \la N_G(x)\rag$ be a 2-connected graph containing two
disjoint pairs of independent vertices. Then $x$ is locally 2-connected in $G$.
\end{lemmaAcite}

Note that if a vertex $x\in V(G)$ is feasible by virtue of 
Theorem~\ref{thmA-HC-2-conn_neighb} (i.e., $x$ is locally 2-connected in $G$), 
then, for any $y\in V(G)$ $y\neq x$, $x$ is locally 2-connected also in $\Gst_y$, 
but $\la N_{\Gst_y}(x)\ra_{\Gst_y}$ can be complete (if $N_G(x)\subset N_G(y)$).
Thus, for any $y\in V(G)$, $x$ is feasible or simplicial in $\Gst_y$. 

We thus define more generally: a set $M\subset V(G)$ is {\em weakly feasible 
in $G$} if the vertices in $M$ can be ordered in a sequence $x_1,\ldots,x_k$ 
such that $x_1$ is feasible in $G_0=G$, and $x_{i+1}$ is feasible or simplicial 
in $G_i=(G_{i-1})^{^*}_{x_i}$, $i=1,\ldots,k-1$.
Thus, similarly, if $G$ is not Hamilton-connected and $M\subset V(G)$ is weakly 
feasible in $G$, then $\Gst_M$ is still not Hamilton-connected and all vertices 
of $M$ are simplicial in $\Gst_M$.

\section{Proof of Theorem~\ref{thm-main}}
\label{sec-proof-main}

In the proof, we will need the following three lemmas. Let $D$ be the diamond
(see Fig.~\ref{fig-diamonds}$(b)$), and $D^1$ and $D^2$ the diamond in which
one or two of the edges $ac_i$, $i=1,2$, are subdivided, respectively
(see Fig.~\ref{fig-diamonds}$(c)$,~$(d)$). We will use the labeling of the
vertices of $D$, $D^1$ and $D^2$ as in Fig.~\ref{fig-diamonds}.

%
%
\begin{lemma}
\label{lemma-diamonds}
Let $G$ be a 3-connected $\{\claw,\Gt\}$-free graph that is not 
Hamilton-connected, let $\bG$ be its $\Gt$-closure, and let $H=\Lp(\bG)$.
Then 
\begin{mathitem}
\item $H$ does not contain as a subgraph the diamond $D$,
\item for any triangle $T\subset H$, every vertex $x\in V(T)$ has a neighbor
      in $V(H)\sm V(T)$,
\item $H$ contains as a subgraph neither the graph $D^1$ such that 
      $N_H(d_1)=\{a,c_1\}$, nor the graph $D^2$ such that 
      $N_H(d_i)=\{a,c_i\}$, $i=1,2$.
\end{mathitem}
\end{lemma}

%
%
\begin{figure}[ht]
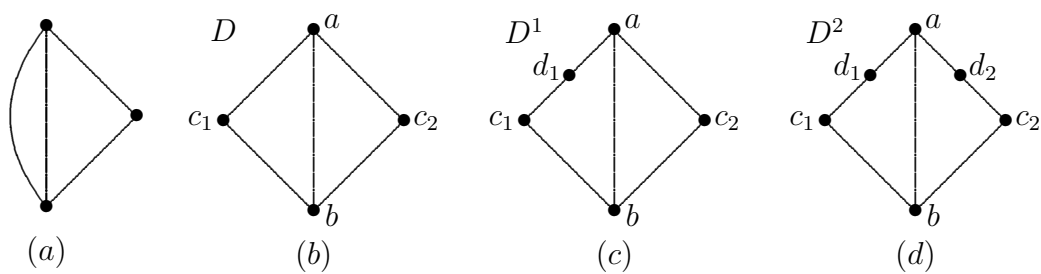

$$\bp
\setcoordinatesystem units <1.0mm,1mm>
\setplotarea x from -50 to 50, y from -7 to 7
\put{\beginpicture
\setcoordinatesystem units <1.2mm,1.2mm>
\setplotarea x from -10 to 10, y from -10 to 10
\put{$\bullet$} at    0   10
\put{$\bullet$} at    0  -10
\put{$\bullet$} at   10    0
\plot 0 10  0 -10  10 0  0 10  0 -10 /
\setquadratic
\plot 0 -10  -4 0  0 10 /
\setlinear
\put{$(a)$} at  0 -15
\endpicture} at  -60   0
\put{\beginpicture
\setcoordinatesystem units <1.2mm,1.2mm>
\setplotarea x from -10 to 10, y from -10 to 10
\put{$\bullet$} at    0   10
\put{$\bullet$} at    0  -10
\put{$\bullet$} at   10    0
\put{$\bullet$} at  -10    0
\put{$a$} at      2   11
\put{$b$} at      2  -10.5
\put{$c_1$} at  -12.5    0
\put{$c_2$} at   12.5    0
\plot 0 10  0 -10  10 0  0 10  -10 0  0 -10 /
\put{$D$} at  -10  10
\put{$(b)$} at  0 -15
\endpicture} at  -25   0
\put{\beginpicture
\setcoordinatesystem units <1.2mm,1.2mm>
\setplotarea x from -10 to 10, y from -10 to 10
\put{$\bullet$} at    0   10
\put{$\bullet$} at    0  -10
\put{$\bullet$} at   10    0
\put{$\bullet$} at  -10    0
\put{$\bullet$} at   -5    5
\put{$a$} at      2   11
\put{$b$} at      2  -10.5
\put{$c_1$} at  -12.5    0
\put{$c_2$} at   12.5    0
\put{$d_1$} at   -7.5    6
\plot 0 10  0 -10  10 0  0 10  -10 0  0 -10 /
\put{$D^1$} at  -10  10
\put{$(c)$} at  0 -15
\endpicture} at   15   0
\put{\beginpicture
\setcoordinatesystem units <1.2mm,1.2mm>
\setplotarea x from -10 to 10, y from -10 to 10
\put{$\bullet$} at    0   10
\put{$\bullet$} at    0  -10
\put{$\bullet$} at   10    0
\put{$\bullet$} at  -10    0
\put{$\bullet$} at   -5    5
\put{$\bullet$} at    5    5
\put{$a$} at      2   11
\put{$b$} at      2  -10.5
\put{$c_1$} at  -12.5    0
\put{$c_2$} at   12.5    0
\put{$d_1$} at   -7.5    6
\put{$d_2$} at    7.5    6
\plot 0 10  0 -10  10 0  0 10  -10 0  0 -10 /
\put{$D^2$} at  -10  10
\put{$(d)$} at  0 -15
\endpicture} at   55   0
\ep$$
\bsm
\caption{The multitriangle, the diamond $D$, and the graphs $D^1$ and $D^2$.}
\label{fig-diamonds}
\end{figure}

\begin{proof}
$(i)$ If $H$ contains a subgraph $F\iso D$, then, since $L(D)=W_4$, $\bG$
is not $W_4$-free, contradicting Proposition~\ref{thmA-closure-W4-W5-P6}.

\ms

$(ii)$ Let $T=u_1u_2u_3$ be a triangle in $H$, denote $e_1=u_1u_2$, 
$e_2=u_2u_3$, $e_3=u_3u_1$, $v_i=L(e_i)$, $i=1,2,3$, and suppose, to the 
contrary, that $N_H(u_1)=\{u_2,u_3\}$.
By Theorem~\ref{thmA-ultim-clos}$(vi)$ and Lemma~\ref{lemmaA-SMclos-degree2},
$\bG$ is not UM-closed, hence some of the vertices $v_1,v_2,v_3$ is feasible
in $\bG$, implying that some of the edges $e_1,e_2,e_3$ is contractible in $H$.
By the definition of the $\Gt$-closure, some $\bG^{^*}_{v_i}$, $i\in\{1,2,3\}$,
contains an induced $\Gt$, i.e., the corresponding subgraph $H|_{e_i}$ contains
an $\Lp(\Gt)$. We will use the labeling of the vertices of $\Gt$ and $\Lp(\Gt)$
as in Fig.~\ref{fig-Preimages}.

By symmetry, it is sufficient to consider the cases when $v_1$ or $v_2$
is feasible.

\begin{mylist}

  \item[\underline{\bf Case 1:}] {\sl $v_1$ is feasible in $\bG$.} \quad \\
  Then the edge $e_1=u_1u_2$ is contractible in $H$, and $H|^{e_1}$ contains a 
  subgraph $F$ such that $L(F)=\Gt$, i.e., $F$ is one of the three graphs in 
  Fig.~\ref{fig-Preimages}.
  
  Let $u_{12}\in V(H|_{e_1})$ be the vertex obtained by identifying $u_1$ and 
  $u_2$. Clearly, $u_{12}\in V(F)$ (otherwise also $F\subset H$), and at least
  one of the edges $u_{12}u_3$ is in $E(F)$ (since $N_H(u_1)=\{u_2,u_3\}$). 
  If $u_{12}u_3=q_iq_{i+1}$ for some $i\in\{1,2,3,4\}$, then also $F\subset H$,
  a contradiction. Hence, up to a symmetry, $u_{12}u_3=s_{12}q_1$.
  Then, replacing in $F$ the edges $(u_{12}u_3)^1,(u_{12}u_3)^2$ by
  $u_1u_3,u_2u_3$ (if $q_1=u_3$), or by $u_1u_2,u_1u_3$ (if $q_1=u_{12}$), 
  we have an $\Lp(\Gt)$ in $H$, a contradiction.
  
  \item[\underline{\bf Case 2:}] {\sl $v_2$ is feasible in $\bG$.} \quad \\
  Then the edge $e_2=u_2u_3$ is contractible in $H$ and $H|_{e_2}$ contains 
  a subgraph $F\iso\Lp(\Gt)$.
  Since $N_{H|_{e_1}}(u_1)=\{u_{23}\}$, by symmetry, $u_{23}\in\{q_1,q_2,q_3\}$.
  
  If $u_{23}=q_1$, then it is straightforward to obtain an $\Lp(\Gt)$ in $H$.
  
  Let $u_{23}=q_2$. Then, replacing in $F$ the edges $q_1q_2$ and $q_1s_1$,
  $q_1s_2$ (or $(q_1s_{12})^1,(q_1s_{12})^2$) by the edges $u_2u_3$, $u_2u_1$, 
  $u_2q_1$ (if $u_3=q_2$), or $u_3u_2$, $u_3u_1$, $u_3q_1$ (if $u_2=q_2$), 
  we have an $\Lp(\Gt)$ in $H$.
  
  Finally, let $u_{23}=q_3$. By symmetry, choose the notation such that, in $H$, 
  $q_2u_2\in E(H)$ and  $q_4u_3\in E(H)$. Since $N_H(u_1)=\{u_2,u_3\}$ and 
  $\{q_2u_2,u_3q_4\}$ cannot be an essential edge-cut in $H$, $u_2$ or $u_3$
  has a neighbor $w\in V(H)\sm (V(F)\cup\{u_1,u_2,u_3\})$. By symmetry, let
  $u_3w\in E(H)$. Then the edges $(s_{12}q_1)^1,(s_{12}q_1)^2$ (or 
  $s_1q_1,s_2q_1$) and $q_1q_2$, $q_2u_2$, $u_2u_1$, $u_1u_3$, $u_3q_4$, $u_3w$
  determine an $\Lp(\Gt)$ in $H$, a contradiction.
\end{mylist}

\bs

$(iii)$ Suppose that $H$ contains a subgraph $F\in\{D^1,D^2\}$ such that 
$N_H(d_1)=\{a,c_1\}$, and if $F=D^2$, then also $N_H(d_2)=\{a,c_2\}$.

\ms 

We first show that, in each of the cases, the subgraph $F$ is contractible to a 
double edge with some pendant edges (i.e., the subgraph $F\subset H$ contains 
some contractible subgraphs such that their contraction turns $F$ into a double 
edge plus some pendant edges). 
This will imply that $L(H)$ is not UM-closed, implying that $L(\tH)$ (where 
$\tH$ is the multigraph obtained from $H$ by the contractions) is still not 
Hamilton-connected. Consequently, $L(\tH)$ must contain an induced $\Gamma_3$,
and we show that this is not possible.

\begin{mylist}

 \item[\underline{\bf Case 1:}] {\sl the edges $ad_1,c_1d_1$, and $ad_2,c_2d_2$
       if $F=D^2$, are simple.} \quad \\
  Then $\co(H)$ contains a diamond, thus, by Theorem~\ref{thmA-ultim-clos}$(vi)$, 
  $L(H)$ is not UM-closed, implying that some edge $f\in E(F)$ is contractible.
   
  \begin{mylist}
   \item[\underline{\bf Subcase 1.1:}] {\sl $F=D^1$.} \quad \\
    First, if $f\in\{ad_1,c_1d_1,c_1b\}$, then $F_1=F|_f$ is a diamond, hence
    $L(ab)$ has a 2-connected neighborhood in $H_1=H|_f$, implying that $ab$ is 
    contractible in $H_1$. Then, in $H_2=H_1|_{ab}$, $F$ contracts to two double
    edges. By Theorem~\ref{thmA-ultim-clos}$(vi)(\beta)$, one of the double edges 
    is contractible, and its contraction yields a double edge (plus some pendant 
    edges). 
    
    Secondly, if $f=ab$, then $F|_f$ is a triangle with a vertex of degree~2
    plus a double edge. By Theorem~\ref{thmA-ultim-clos}, some of the edges of 
    the triangle is contractible and the contraction yields a double edge (plus 
    some pendant edges).
    
    Finally, if $f\in\{ac_2,bc_2\}$, then $F_1=F|_f$ is the graph with edges
    $ad_1,c_1d_1,c_1b$ and a double edge $ab$, and since $d_1$ is of degree~2, 
    $F_1$ corresponds to a multitriangle in $\co(H|_f)$. Thus, by 
    Theorem~\ref{thmA-ultim-clos}$(vi)$, some edge $f_1\in E(F_1)$ is 
    contractible in $H_1=H|_f$. 
    If $f_1\in\{ad_1,c_1d_1,c_1b\}$, then $F_2=F_1|_{f_1}$ is a multitriangle,
    $ab$ is contractible in $H_2=H_1|{f_1}$ (since it has 2-connected 
    neighborhood), and $F_2|_{ab}$ yields a double edge. 
    If $f_1=ab$, then $F_2=F_1|_{ab}$ is a triangle with a vertex of degree~2, 
    and a contraction of any of its edges (by Theorem~\ref{thmA-HC-2-conn_neighb})
    yields again a double edge (plus pendant edges).

   \item[\underline{\bf Subcase 1.2:}] {\sl $F=D^2$.} \quad \\
    If $f\in E(F)\sm\{ab\}$, then $F_1=F|_f\iso D^1$ and we are in some of the 
    previous cases. Thus, let $f=ab$. Then $F_1=F|_f$ consists of two triangles,
    each with a vertex of degree~2, which by Theorem~\ref{thmA-ultim-clos} and 
    Lemma~\ref{lemmaA-SMclos-degree2} yields again a double edge.
    
 \end{mylist}

 \item[\underline{\bf Case 2:}] {\sl some of the edges $ad_1,c_1d_1$, and 
      $ad_2,c_2d_2$ if $F=D^2$, is multiple in $H$.} \quad \\
  In this case, we have the following fact.

 %
 %
 \begin{claimbc}
 \label{claim-lemma-diamond}
 Each of the edges $ad_1,c_1d_1$, and $ad_2,c_2d_2$ if $F=D^2$, that is 
 multiple in $H$, is contractible.
 \end{claimbc}

\bsm

 \begin{proofcl} 
 Suppose that some $f\in\{ad_1,c_1d_1\}$ is multiple but not contractible.

 First observe that the edge $ab$ is not contractible, since otherwise,
 in $F|_{ab}$, the edges $ad_1,c_1d_1$ and $c_1b$ create a multitriangle 
 in which $f$ is in a multiple edge, hence $L(f)$ has 2-connected neighborhood 
 and $f$ is contractible, a contradiction. Thus, $ab$ is not contractible, 
 specifically, $ab$ is simple.
 
 If both $ad_1$ and $c_1d_1$ are multiple, then one of them is contractible 
 (by Theorem~\ref{thmA-ultim-clos}$(vi)$), and its contraction yields again a 
 multitriangle with $f$ in multiple edge, implying contractibility of $f$, a 
 contradiction. Thus, exactly one of the edges $ad_1$, $c_1d_1$ is multiple.

 Now, if $F=D^1$, then $abc_2$ is a triangle, and since $f$ is a 
 noncontractible multiedge, some of the edges $ac_2$, $bc_2$ is contractible,
 implying contractibility of $ab$, a contradiction. Thus, $F=D^2$. 

 If none of the edges $ad_2,c_2d_2$ is multiple, then $abc_2$ is a triangle
 in $\co(H)$. Since $f$ is a noncontractible multiedge, some of the edges 
 $ad_2,c_2d_2$ is contractible, implying again contractibility of $ab$,
 a contradiction. Thus, some edge $f_1\in\{ad_2,c_2d_2\}$ is multiple. 
 But then again, by Theorem~\ref{thmA-ultim-clos}$(vi)$, $f_1$ is 
 contractible, implying, as before, contractibility of $ab$, a contradiction.
 \end{proofcl} 

\vspace{-5mm}

 We summarize that some edges of $F$ are multiple, and each multiple edge 
 of $F$ is contractible. But now, contracting some of the multiple edges 
 of $F$, in each of the cases, we are in some of the previous cases.
 Thus, we conclude that $F$ can be contracted to a double edge plus some 
 pendant edges, and for the resulting multigraph $\tH$, $L(\tH)$ is still
 not Hamilton-connected. 
\end{mylist}

\bs 

By the definition of the $\Gt$-closure, $\tH$ contains a subgraph 
$\tF\iso\Lp(\Gt)$. 
Moreover, observe that, in each of the cases, the contraction of $F$ to a double
edge contracts some of the two ``subdivided triangles" of $F$ to a vertex 
(plus some pendant edges).
More specifically, either the subgraph $F^1$ with edges $ad_1,d_1c_1,c_1b$ and 
$ab$, or the subgraph $F^2$ with edges $ad_2,d_2c_2$ (or $ac_2$ if $F=D^1$), 
$c_2b$ and $ab$, is contracted to a vertex plus pendant edges, and this turns $F$
into a double edge (plus some pendant edges).
Thus, at most one of the edges of $F$ is an edge of $\tF$.
Denote the vertices of $\tF$ as in Fig.~\ref{fig-Preimages}.

We will consider the vertex of $\tH$ resulting from the contraction 
of $F^1$ or $F^2$ as a 
``new" vertex, and will distinguish cases according to which of the vertices 
of $\tF$ is new, and, subject to this, which of the edges $q_iq_{i+1}$,
$i=1,2,3,4$, is an edge of $F$. 
To reach a contradiction, in each of the cases, we will list edges of an 
$\Lp(\Gt)$ in $H$.
We will also list which part of $F$ (i.e., $F^1$ or $F^2$) corresponds
to the new vertex, and its vertices that are used in the $\Lp(\Gt)$ in $H$.
The cases and subcases are distinguished up to a symmetry, and, in symmetric 
situations, we will always list the possibility that $F^1$ is contracted
(note that if $F=D^1$, it is possible that the $\Lp(\Gt)$ in $H$ uses the edge
$ad_1$ or $c_1d_1$ when $F^1$ is contracted, while it uses the edge $ab$ 
if $F^2$ is contracted, and we consider these situations also symmetric).
In all cases, when the edges $q_1s_1$ and $q_1s_2$, or $q_5s_3$ and $q_5s_4$
are used, it is always implicitly understood that there is also a possibility 
of a double edge $q_1s_{12}$ or $q_5s_{34}$.

\begin{mylist}

 \item[\underline{\bf Case 1:}] {\sl $q_1$ is new.} \quad \\
 Then possibly $E(F)\cap E(\tF)=\emptyset$ or $q_1q_2\in E(F)$, and (up to a 
 symmetry) $q_1\in\{c_1,b\}$
 if $q_1q_2\notin E(F)$, and $q_1=b$ if $q_1q_2\in E(F)$. We thus have the 
 following possibilities.

 \begin{tabular}{|c|l|l|}
 \hline 
 \rule{0pt}{3ex}
 $E(F)\cap E(\tF)$ & ~~~$q_1$ & edges of an $\Lp(\Gt)$ in $H$ \\
 \hline
 $\emptyset$ & $F^1$; $c_1$ & $c_1d_1,c_1b,c_1q_2,q_2q_3,q_3q_4,q_4q_5,q_5s_3,q_5s_4$ \\
 $\emptyset$ & $F^1$; $b$   & $bc_1,ba,bq_2,q_2q_3,q_3q_4,q_4q_5,q_5s_3,q_5s_4$  \\
 $q_1q_2=bc_2$ & $F^1$; $b$ & $bc_1,ba,bc_2,c_2q_3,q_3q_4,q_4q_5,q_5s_3,q_5s_4$  \\
 \hline
 \end{tabular}

 \item[\underline{\bf Case 2:}] {\sl $q_2$ is new.} \quad \\
 In this case, possibly  $E(F)\cap E(\tF)=\emptyset$, $q_1q_2\in E(F)$, or 
 $q_2q_3\in E(F)$, and we have the following possibilities.
  
 \begin{tabular}{|c|l|l|}
 \hline
 \rule{0pt}{2.5ex}
 $E(F)\cap E(\tF)$ & ~~~~$q_2$ & edges of an $\Lp(\Gt)$ in $H$ \\
 \hline
 $\emptyset$ & $F^1$; $a,b$ & $bc_1,bq_1,ba,aq_3,q_3q_4,q_4q_5,q_5s_3,q_5s_4$ \\
 $q_1q_2=bc_2$ & $F^1$; $b,c_1$ & $bc_2,ba,bc_1,c_1q_3,q_3q_4,q_4q_5,q_5s_3,q_5s_4$ \\
 $q_1q_2=bc_2$ & $F^1$; $b$ & $ad_1,ad_2(ac_2),ab,bq_3,q_3q_4,q_4q_5,q_5s_3,q_5s_4$ \\
 $q_2q_3=bc_2$ & $F^1$; $b,c_1$ & 
                    $c_1d_1,c_1q_1,c_1b,bc_2,c_2q_4,q_4q_5,q_5s_3,q_5s_4$ \\
 $q_2q_3=bc_1$ & $F^2$; $a,b$ & $aq_1,ab,ad_1,d_1c_1,c_1q_4,q_4q_5,q_5s_3,q_5s_4$ \\
 \hline
 \end{tabular}

 \item[\underline{\bf Case 3:}] {\sl $q_3$ is new.} \quad \\
 In this case, possibly  $E(F)\cap E(\tF)=\emptyset$, $q_2q_3\in E(F)$, or 
 $q_3q_4\in E(F)$; however, the last two possibilities are symmetric and we 
 therefore  consider only the first of them.

 \begin{mylist}
  \item[\underline{\bf Subcase 3.1:}] {\sl $q_2q_3\in E(F)$.} \quad \\
  Then we have the following possibilities.
  
  \begin{tabular}{|c|l|l|}
  \hline
 \rule{0pt}{2.5ex}
  $E(F)\cap E(\tF)$ & ~~~~$q_2$ & edges of an $\Lp(\Gt)$ in $H$ \\
  \hline
  $q_2q_3=ac_2$ & 
      $F^1$; $a,c_1$ &$c_2q_1,c_2d_2(c_2b),c_2a,ac_1,c_1q_4,q_4q_5,q_5s_3,q_5s_4$ \\
  $q_2q_3=c_1b$ & $F^2$; $a,b$ &$c_1q_1,c_1b,c_1d_1,d_1a,aq_4,q_4q_5,q_5s_3,q_5s_4$ \\
  \hline
  \end{tabular}

  \ms
  
  \item[\underline{\bf Subcase 3.2:}] {\sl $E(F)\cap E(\tF)=\emptyset$.} \quad \\
  Since one of $F^1,F^2$ is contracted and $d_1,d_2$ have no neighbors outside 
  $F$, the vertices $q_2$ and $q_4$ are adjacent in $H$ either to $a$ and $b$, 
  or to one of $a,b$ and one of $c_1,c_2$. So, up to a symmetry, either 
  $q_2a,q_4b\in E(H)$, or $N_H(q_2)\cap\{c_1,c_2\}\neq\emptyset$ and 
  $N_H(q_4)\cap\{a,b\}\neq\emptyset$.

  \begin{mylist}

   \item[\underline{\bf Subcase 3.2.1:}] {\sl $q_2a,q_4b\in E(H)$.} \quad \\
   If $F=D^1$, then the edges $q_1s_1,q_1s_2,q_1q_2,q_2a,ac_2,c_2b,bq_4,bc_1$ 
   determine an $\Lp(\Gt)$ in~$H$, a contradiction. Hence $F=D^2$.
   
   If $bc_1$ is a double edge in $H$, then the edges 
   $q_1s_1,q_1s_2,q_1q_2,q_2a,ad_1,d_1c_1,(c_1b)^1,(c_1b)^2$, 
   and if $ad_1$ is double in $H$, then the edges 
   $(ad_1)^1,(ad_1)^2,d_1c_1,c_1b,bq_4,q_4q_5,q_5s_3,q_5s_4$
   determine an $\Lp(\Gt)$ in $H$. Thus, $\mu_H(ad_1)=\mu_H(bc_1)=1$.
   Since $\{ad_1,bc_1\}$ cannot be an essential edge-cut, $c_1$ has another 
   neighbor $z\in V(H)$. We show that $z$ cannot be any of the vertices 
   $s_1,s_2$ (or $s_{12}$), $q_1,q_2,a$.
  
   \begin{tabular}{|c|l|}
   \hline
   \rule{0pt}{2.5ex}
   $z$ & edges of an $\Lp(\Gt)$ in $H$ \\
   \hline
   $s_1$ & $q_1s_2,q_1q_2,q_1s_1,s_1c_1,c_1b,ba,ad_1,ad_2$ \\
   $s_{12}$ & 
        $(s_{12}q_1)^1,(s_{12}q_1)^2,s_{12}c_1,c_1b,bq_4,q_4q_5,q_5s_3,q_5s_4$ \\
   $ q_1$ & $q_1s_1,q_1s_2,q_1c_1,c_1b,bq_4,q_4q_5,q_5s_3,q_5s_4$ \\
   $ q_2$ & $q_2q_1,q_2a,q_2c_1,c_1b,bq_4,q_4q_5,q_5s_3,q_5s_4$ \\
   $ a$   & $aq_2,ad_1,ac_1,c_1b,bq_4,q_4q_5,q_5s_3,q_5s_4$ \\
  \hline
  \end{tabular}

  But then the edges $q_1s_1,q_1s_2,q_1q_2,q_2a,ad_1,d_1c_1,c_1z,c_1b$
  determine an $\Lp(\Gt)$ in $H$, a contradiction.

  \item[\underline{\bf Subcase 3.2.2:}] {\sl $N_H(q_2)\cap\{c_1,c_2\}\neq\emptyset$ 
  and $N_H(q_4)\cap\{a,b\}\neq\emptyset$.} \quad \\
  Let $\bq_2\in\{c_1,c_2\}$ and $\bq_4\in\{a,b\}$ denote the neighbor of $q_2$
  or $q_4$ in $F$, respectively. 
   
  Suppose first that $\bq_2=c_1$. Then either $\bq_4=a$, or $\bq_4=b$.
   
  \begin{tabular}{|c|l|}
  \hline
  \rule{0pt}{2.5ex}
  $\bq_4$ & edges of an $\Lp(\Gt)$ in $H$ \\
  \hline
  $a$ & $q_1s_1,q_1s_2,q_1q_2,q_2c_1,c_1d_1,d_1a,ab,aq_4$ \\
  $b$ & $q_1s_1,q_1s_2,q_1q_2,q_2c_1,c_1d_1,d_1a,ab,ad_2(ac_2)$ \\
  \hline
  \end{tabular}   

  We get a symmetric contradiction if $\bq_2=c_2$ and $F=D^2$. Thus, we have 
  $\bq_2=c_2$ and $F=D^1$ (and either $\bq_4=a$, or $\bq_4=b$).
   
  We now show that none of the edges $ad_1,bc_1$ can be a double edge.
  Suppose the opposite. Then we have the following possibilities.
      
  \begin{tabular}{|c|l|}
  \hline
  \rule{0pt}{2.5ex}
  double edge & edges of an $\Lp(\Gt)$ in $H$ \\
  \hline
  $ad_1$ & $q_1s_1,q_1s_2,q_1q_2,q_2c_2,c_2b,ba,(ad_1)^1,(ad_1)^2$ \\
  $bc_1$ & $q_1s_1,q_1s_2,q_1q_2,q_2c_2,c_2a,ab,(bc_1)^1,(bc_1)^2$ \\
  \hline
  \end{tabular}     
   
  Thus, $\mu_H(ad_1)=\mu_H(bc_1)=1$. Since $\{ad_1,bc_1\}$ cannot be an 
  essential edge-cut, $c_1$ must have another neighbor $z\in V(H)$. 
  Clearly $z\neq c_2$ (since then we would have a diamond in $H$). 
  We show that $z$ cannot be any of the vertices $s_1,s_2$ (or $s_{12}$), 
  $q_1,q_2$.
  
   \begin{tabular}{|c|l|}
   \hline
   \rule{0pt}{2.5ex}
   $z$ & edges of an $\Lp(\Gt)$ in $H$ \\
   \hline
   $s_1$ & $q_1s_2,q_1q_2,q_1s_1,s_1c_1,c_1b,ba,ad_1,ac_2$ \\
   $s_{12}$ & 
        $(s_{12}q_1)^1,(s_{12}q_1)^2,s_{12}c_1,c_1d_1,d_1a,ac_2,c_2q_2,c_2b$ \\
   $ q_1$ & $q_1s_1,q_1s_2,q_1c_1,c_1d_1,d_1a,ac_2,c_2q_2,c_2b$ \\
   $ q_2$ & $q_1s_1,q_1s_2,q_1q_2,q_2c_1,c_1d_1,d_1a,ac_2,ab$ \\
  \hline
  \end{tabular}

  But then the edges $q_1s_1,q_1s_2,q_1q_2,q_2c_2,c_2b,bc_1,c_1d_1,c_1z$ 
  determine an $\Lp(\Gt)$ in $H$, a contradiction.
  \end{mylist}
 \end{mylist}
\end{mylist}

\bsm\bsm
\end{proof}

\bs

The next lemma will describe some structural properties of a minimal 
counterexample to Theorem~\ref{thm-main}.
Here, we say that a graph is {\em minimal} with respect to a property $\cal{P}$, 
if $G$ has $\cal{P}$, but for every vertex $x\in V(G)$, $G-x$ does not have 
$\cal{P}$. 

For an integer $r\geq 2$, $\dvK_{2,r}^M$ will denote the family of multigraphs 
that can be obtained from the complete bipartite graph 
$K_{2,r}=(\{v_1,v_2\},\{w_1,w_2,\ldots,w_r\})$ by replacing at least one 
of the edges $w_iv_1,w_iv_2$ with a double edge
(for an example, see Fig.~\ref{fig-multidiamonds}$(b)$).
Similarly, $\dvK_{rP_4}^M$ will denote the family of multigraphs obtained by 
identifying endvertices of $r$ vertex-disjoint paths $P_4^i=v_1w_iz_iv_2$, 
$i=1,2,\ldots,r$, and, for each $i=1,2,\ldots,r$, by replacing at least 
one of the edges of the $P_4^i$ with a double edge (for an example, see 
Fig.~\ref{fig-multidiamonds}$(c)$).

%
%
\begin{figure}[ht]
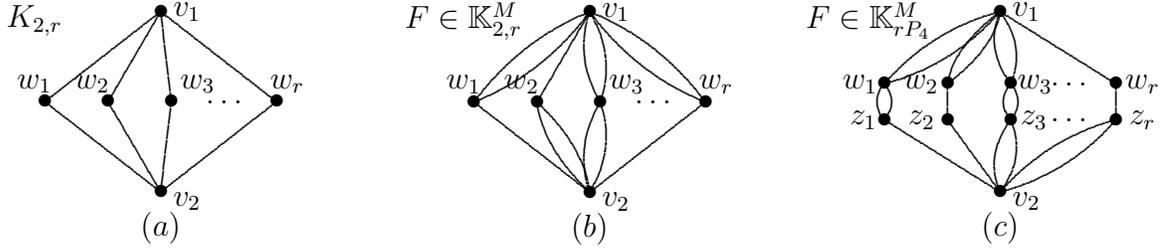

$$\bp
\setcoordinatesystem units <1.1mm,1mm>
\setplotarea x from -70 to 70, y from -7 to 7
\put{\beginpicture
\setcoordinatesystem units <1.4mm,1.2mm>
\setplotarea x from -10 to 10, y from -10 to 10
\put{$\bullet$} at    0  -10
\put{$\bullet$} at    0   10
\put{$\bullet$} at  -11    0
\put{$\bullet$} at   -5    0
\put{$\bullet$} at    1    0
\put{$\ldots$} at     6    0
\put{$\bullet$} at   11    0
%
%
\put{$v_1$} at     2.5   10
\put{$v_2$} at     2.5  -11
\put{$w_1$} at   -12      2
\put{$w_2$} at    -6.5    2
\put{$w_3$} at     3.5    2
\put{$w_r$} at    12      2
\plot  0 -10   -11 0   0 10  /
\plot  0 -10   -5 0   0 10  /
\plot  0 -10   1 0   0 10  /
\plot  0 -10   11 0   0 10  /
\put{$K_{2,r}$} at  -12  9
\put{$(a)$} at  0 -14
\endpicture} at   -50   0
\put{\beginpicture
\setcoordinatesystem units <1.4mm,1.2mm>
\setplotarea x from -10 to 10, y from -10 to 10
\put{$\bullet$} at    0  -10
\put{$\bullet$} at    0   10
\put{$\bullet$} at  -11    0
\put{$\bullet$} at   -5    0
\put{$\bullet$} at    1    0
\put{$\ldots$} at     6    0
\put{$\bullet$} at   11    0
%
%
\put{$v_1$} at     2.5   10
\put{$v_2$} at     2.5  -11
\put{$w_1$} at   -12      2
\put{$w_2$} at    -6.5    2
\put{$w_3$} at     3.5    2
\put{$w_r$} at    12      2
\plot  -11 0  0 -10   11 0  /
\plot   -5 0  0  10   /
%
\setquadratic
\plot   1 0   1.5  5  0  10  -0.5  5    1 0 /
\plot   1 0   1.5 -5  0 -10  -0.5 -5    1 0 /
\plot  11 0   5    4  0  10   6    6   11 0 /
\plot -11 0  -5    4  0  10  -6    6  -11 0 /
\plot  -5 0  -3.5 -5  0 -10  -1.5 -5   -5 0 /
\setlinear
\put{$F\in\dvK_{2,r}^M$} at  -12  9
\put{$(b)$} at  0 -14
\endpicture} at   0   0
\put{\beginpicture
\setcoordinatesystem units <1.4mm,1.2mm>
\setplotarea x from -10 to 10, y from -10 to 10
\put{$\bullet$} at    0  -10
\put{$\bullet$} at    0   10
\put{$\bullet$} at  -11    2
\put{$\bullet$} at   -5    2
\put{$\bullet$} at    1    2
\put{$\ldots$} at     6.5  2
\put{$\bullet$} at   11    2
\put{$\bullet$} at  -11   -2
\put{$\bullet$} at   -5   -2
\put{$\bullet$} at    1   -2
\put{$\ldots$} at     6.5 -2
\put{$\bullet$} at   11   -2
\put{$v_1$} at     2.5   10
\put{$v_2$} at     2.5  -11
\put{$w_1$} at   -13.3    2
\put{$w_2$} at    -7.5    2
\put{$w_3$} at     3.3    2
\put{$w_r$} at    13.5    2
\put{$z_1$} at   -13     -2
\put{$z_2$} at    -7.5   -2
\put{$z_3$} at     3.2   -2
\put{$z_r$} at    13.5   -2
\plot 0 -10 -11 -2 /
\plot 0 -10  -5 -2   -5  2 /
\plot  11 -2   11  2  0 10 /
\setquadratic
\plot   1  2   1.5  6  0  10  -0.5  6    1  2 /
\plot   1 -2   1.5 -6  0 -10  -0.5 -6    1 -2 /
\plot -11  2  -5    5  0  10  -6    7  -11  2 /
\plot  11 -2   5   -5  0 -10   6   -7   11 -2 /
\plot  -5 2  -3.5  6  0  10  -1.5  6   -5 2 /
\plot -11 -2  -11.7 0  -11 2  -10.3 0  -11 -2 /
\plot 1 -2  1.7 0  1 2  0.3 0  1 -2 /
\setlinear
\put{$F\in\dvK_{rP_4}^M$} at  -12.5  9
\put{$(c)$} at  0 -14
\endpicture} at   50   0
\ep$$
\bsm
\caption{The graph $K_{2,r}$, a multigraph from $\dvK_{2,r}^M$ and a multigraph 
         from $\dvK_{rP_4}^M$.}
\label{fig-multidiamonds}
\end{figure}
%

%
%
\begin{lemma}
\label{lemma-minimality}
Let $G$ be a minimal 3-connected $\{\claw,\Gt\}$-free non-Hamilton-connected
graph, let $\bG$ be its $\Gt$-closure, and let $H=\Lp(\bG)$.
Then 
\begin{mathitem}
\item every vertex $x\in V(H)$ is incident with at most two pendant
      edges,
\item every edge $e\in E(H)$ that is in a cycle of length at least three 
      has multiplicity $\mu(e)\leq 2$,
\item $H$ does not contain as a subgraph the graph $K_{2,4}$ such that its 
      vertices of degree~2 are of degree~2 in $H$,
\item $H$ does not contain as a subgraph a multigraph from $\dvK_{2,4}^M$ 
      such that every vertex of $\bigcup_{i=1}^4 N_H(w_i) \sm \{v_1,v_2\}$ 
      has degree $1$ in $H$,
\item $H$ does not contain as a subgraph a multigraph from $\dvK_{4P_4}^M$ 
      such that $N_H(w_i)=\{v_1,z_i\}$ and $N_H(z_i)=\{v_2,w_i\}$, 
      $i=1,2,3,4$.
\end{mathitem}
\end{lemma}

\begin{proof}
$(i)$ By the assumption of the lemma, $H$ is essentially 3-edge-connected,
not containing an $\Lp(\Gt)$, and $H$ does not have an $(e,f)$-IDT for some 
$e,f\in E(H)$. 
Suppose that $H$ contains $s\geq 3$ pendant edges at a vertex $x\in V(H)$, 
choose $s-2$ of them such that none of them is $e$ or $f$, and let $H'$ be 
obtained from $H$ by removing the chosen $s-2$ pendant edges. Then 
clearly $H'$ is still essentially 3-edge-connected, not containing an 
$\Lp(\Gt)$, and $H'$ does not have an $(e,f)$-IDT. Since 
$|V(L(H'))|=|V(G)|-(s-2)<|V(G)|$, we have a contradiction with the minimality
of $G$.

\ms

$(ii)$ Similarly, suppose that $\mu_H(e)=s\geq 3$ for some edge $e=ab\in E(H)$, 
let $e,f\in E(H)$ be such that $H$ has no $(e,f)$-IDT, choose the notation
such that none of the edges $(ab)^3,\ldots,(ab)^s$ is any of $e,f$, and let 
$H'$ be obtained by from $H$ by removing the edges $(ab)^3,\ldots,(ab)^s$.
Then clearly $H'$ contains no $\Lp(\Gt)$, no $(e,f)$-IDT, and since $e=ab$ is
in a cycle of length at least~3, $H'$ is also essentially 3-edge-connected. 
Thus, the graph $G'=L(H')$ contradicts the minimality of $G$.

\ms

$(iii)$ 
Suppose that $H$ contains a subgraph 
$F\iso K_{2,4}=(\{v_1,v_2\},\{w_1,w_2,w_3,w_4\})$ such that $d_H(w_i)=2$, 
$i=1,2,3,4$. Then in $\co(H)$ we have $\mu(v_1v_2)=4$, hence $H$ is not
UM-closed by Theorem~\ref{thmA-ultim-clos}$(vi)$. Thus, some of the edges 
$w_iv_j$ is contractible. Choose the notation such that $v_1w_1$ is 
contractible and set $H_1=H|_{v_1w_1}$ and $G_1=L(H_1)$.
Then, in $G_1$, $\la N_{G_1}(L(v_1v_2))\ra_{G_1}$ is 2-connected by
Lemma~\ref{lemmaA-2_indep_sets}, hence $v_1v_2$ is contractible by 
Theorem~\ref{thmA-HC-2-conn_neighb}.
Set $H_2=H_1|_{v_1v_2}$ and $G_2=L(H_2)$. Then, in $H_2$, the whole $F$
contracts to a single vertex plus 8 pendant edges, and $G_2$ is still 
3-connected and not Hamilton-connected.

Let now $H'$ be obtained from $H$ by replacing the subgraph $F$ with the graph 
$F'=K_{2,3}=(\{v_1,v_2\},\{w_1,w_2,w_3\})$. Then, analogously, $v_1v_2$ is a 
triple edge in $\co(H')$, $v_1w_1$ (say) is contractible in $H'$, and 
$\la N_{G'_1}(L(v_1v_2))\ra_{G'_1}$ is 2-connected in $G'_1=L(H'_1)$, where 
$H'_1=H'|_{v_1w_1}$. In $H'_2=H'_1|_{v_1v_2}$ then $F'$ contracts to a single 
vertex plus 6 pendant edges. Moreover, $G'=L(H')$ is 3-connected, and $G'$ is 
Hamilton-connected if and only if $G'_2=L(H'_2)$ is Hamilton-connected. 

Since $H_2$ and $H_2'$ differ only in number of pendant edges at the vertex
resulting from contracting $F$ or $F'$, respectively, $G_2'$ is also not 
Hamilton-connected, implying $G'$ is not Hamilton-connected. Since $G'$
is an induced subgraph of $G$, $G'$ is $\Gt$-free. Thus, $G'$ contradicts 
the minimality of $G$.

\ms 

$(iv)$ 
Suppose that $H$ contains a submultigraph $F\in\dvK_{2,4}^M$ satisfying the 
assumptions of the lemma.
Note that, by part $(ii)$, every edge of $F$ has multiplicity 1 or 2, and,
by the definition of $\dvK_{2,4}^M$, at least one of the edges $w_iv_1,w_iv_2$ 
has multiplicity~2, $i=1,2,3,4$. For each $i=1,2,3,4$, let $f_i$ be one of 
the edges $w_iv_1,w_iv_2$ with $\mu(f_i)=2$. 

By Theorem~\ref{thmA-ultim-clos}$(vi)(\beta)$, some three of the edges 
$f_1,f_2,f_3,f_4$ are contractible, and we choose the notation such that 
the contractible edges are $f_2,f_3$ and $f_4$. 

Let $H'$ be the multigraph obtained from $H$ by removing the vertex $w_4$ 
and possibly the (by $(i)$ at most two) its neighbors of degree~1.
Observe that $H'$ is also essentially 3-edge-connected.

Set $H_1=H|_{\{f_2,f_3,f_4\}}$ and $H'_1=H'|_{\{f_2,f_3\}}$.
Then in $H$ the submultigraph $F$ contracts to a graph with vertices 
$v_1,v_2,w_1$, one multiple edge $v_1v_2$, and two edges $w_1v_1$ and $w_1v_2$, 
at least one of them being multiple. 
In $H'$, $F'$ contracts to the same structure with the only difference 
that $\mu_{H'}(v_1v_2)<\mu_{H}(v_1v_2)$. More specifically, since the multiple 
edge $v_1v_2$ results from 3 edges in $H$ but 2 edges in $H'$, we have
$\mu_H(v_1v_2)\geq\mu_{H'}(v_1v_2)+1$.

Again by Theorem~\ref{thmA-ultim-clos}$(vi)(\beta)$, one of the multiple edges 
in contractible, resulting in an edge with multiplicity at least~3 (both in $H$
and in $H'$), and another application of Theorem~\ref{thmA-ultim-clos}$(vi)$
contracts the whole $F$ in $H$ (or $F'$ in $H'$) to a single vertex with some 
pendant edges. Thus, the multigraphs $H|_F$ and $H'|_{F'}$  are isomorphic 
up to a different number of pendant edges at the vertex resulting from 
contracting $F$ (or $F'$). Consequently, $L(H|_F)$ is Hamilton-connected 
if and only if $L(H'|_{F'})$ is Hamilton-connected, implying that $G'$ is not
Hamilton-connected. Since $H'$ is essentially 3-edge-connected, $G'$ is 
3-connected, and $G'$ is $\Gt$-free since $G'\indsub G$. 
Thus, the graph $G'$ contradicts the minimality of $G$.

\ms

$(v)$
Suppose that $H$ contains a subgraph $F\in\dvK_{4P_4}^M$.
First observe that if, say, $w_1z_1$ is a double edge, then the edges 
$(w_1z_1)^1,(w_1z_1)^2,w_1v_1,v_1w_2,w_2z_2,z_2v_2,v_2z_3,v_2z_4$
determine an $\Lp(\Gt)$ in H, a contradiction. Thus, by symmetry, 
all edges $w_iz_i$, $i=1,2,3,4$, are simple edges.

By Theorem~\ref{thmA-ultim-clos}$(vi)(\beta)$, all multiple edges in $F$, 
except for possibly one, are contractible. Choose the notation such that 
the (possibly) noncontractible double edge is the edge $v_1w_1$, and let
$\tF$ be the subgraph of $F$ consisting of $P_4^2$, $P_4^3$ and $P_4^4$. 
Let $F_c$ be the set of contractible edges of $F$, and set $H_1=H|_{F_c}$. 
Then, in $H_1$, each of the paths $P_4^i$, $i=2,3,4$, contracts either 
to an edge $v_1v_2$ plus 4 pendant edges (if $|E(P_4^i)\cap F_c|=2$), 
or to a $(v_1,v_2)$-path of length~2 with interior vertex of degree~2,
plus 2 pendant edges (if $|E(P_4^i)\cap F_c|=1$).
Thus, the subgraph $\tF_1=\tF|_{F_c}\subset H_1$ corresponds in $\co(H_1)$ 
to a triple edge $v_1v_2$. By Theorem~\ref{thmA-ultim-clos}$(vi)$, 
$H_1$ is not UM-closed, hence some of the edges of $\tF_1$, say, $e$, is 
contractible.

If $e=v_1v_2$, then in $H_2=H_1|_e$ the whole $\tF$ contracts to a single 
vertex plus some pendant edges. Otherwise, $e$ is an edge of a $(v_1,v_2)$-path 
of length~2, and then in $H_2=H_1|_e$ the $(v_1,v_2)$-path is replaced with 
an edge $v_1v_2$ (plus a pendant edge); and repeating the argument, we are 
in the first case. Thus, in each of the cases, the contractions result in 
the graph $H_2$, in which $\tF$ is contracted to a single vertex (plus pendant 
edges). 
The graph $G_2=L(H_2)$ is 3-connected since clearly $H_2$ is essentially
3-edge-connected, and, by Theorem~\ref{thmA-ultim-clos}$(iii)$, $G_2$
is not Hamilton-connected. 

Let now $H'$ be obtained from $H$ by replacing the subgraph $F$ with 
$F'\in\dvK_{3P_4}^M$, and choose again the notation such that the (possibly)
noncontractible multiedge is the double edge $v_1w_1$. 
Let $\tF'$ be the subgraph of $F'$ consisting of $P_4^2$ and $P_4^3$. 
Then clearly $H'$ is also essentially 3-edge-connected, hence $G'=L(H')$
is 3-connected. Moreover, $H'|_{F'}$ is the same multigraph as $H_2$, with only
different number of pendant edges at the vertex resulting from contracting 
$\tF$ (or $\tF'$, respectively). Consequently, $L(H'|_{F'})$ is not 
Hamilton-connected, hence $G'=L(H')$ is also not Hamilton-connected (since
$L(H|_{F'})$ was obtained from $G'$ by a series of local completions).
Since $G'$ is an induced subgraph of $G$, the graph $G'$ is $\Gt$-free,
hence $G'$ contradicts the minimality of $G$.
\end{proof}

\ms

The following lemma will be crucial in the proof of Theorem~\ref{thm-main}
for graphs containing a small cycle. In the lemma, $C_9^M$ denotes the multigraph
obtained from the cycle $C_9=x_0x_1\ldots x_8$ by adding one parallel edge 
to each of the edges $x_0x_1$, $x_3x_4$ and $x_6x_7$, and $\cF$ denotes the 
finite family consisting of all multigraphs listed in the file {\tt F.txt}
available at \cite{computing1}.
We will folow the labeling of vertices of some special graphs as introduced in 
Figures~\ref{fig-diamonds} and \ref{fig-multidiamonds}.

%
%
\begin{lemma}
\label{lemmaX}
Let $G$ be a $\Gamma_3$-free line graph of a multigraph and let $H=L^{-1}(G)$. 
Then $H \in \mathcal{F}$ if and only if
$H$ satisfies all conditions (1), \dots, (8) 
and every subgraph $F$ of $H$
satisfies each of conditions (9), \dots, (14):
\begin{enumerate}
 \item[(1)] each vertex of $H$ has at most two neighbors of degree 1,
 \item[(2)] each multiedge of $H$ has multiplicity at most 2,
 \item[(3)] $H$ contains $C_k$ as a subgraph for some $k \in \{7, \dots, 10\}$,
 \item[(4)] $H$ does not contain $D$ as a subgraph,
 \item[(5)] $H$ does not contain $C_9^M$ as a flat subgraph,
 \item[(6)] at least 10 vertices of $H$ have degree at least 3,
 \item[(7)] $H$ is essentially $3$-edge-connected,
 \item[(8)] $H$ is essentially $2$-connected,
 \item[(9)] if $F\iso K_3$, then every vertex of $F$ has at least three neighbors 
    in $H$,
 \item[(10)] if $F\iso D^1$, then every vertex of $F$ has at least three neighbors 
    in $H$,
 \item[(11)] if $F\iso K_{2,4}$, then $\sum_{i=1}^4 |N_H(w_i)| > 8$,
 \item[(12)] if $F\iso D^2$, then $|N_H(c_1)|+|N_H(c_2)| > 4$ and 
     $|N_H(d_1)|+|N_H(d_2)| > 4$,
 \item[(13)] if $F\in \mathbb{K}^M_{4,P_4}$, then 
     $\sum_{u\in V(F)\sm\{v_1,v_2\}}|N_H(u)| > 16$,
 \item[(14)] if $F\in \mathbb{K}^M_{2,4}$, then some vertex of 
     $\cup^4_{i=1} N_H(w_i)\sm\{v_1,v_2\}$ has degree at least $2$ in~$H$.
\end{enumerate}
\end{lemma}

\begin{proof}
We prove the lemma with the help of a computer.
To this end, we design an algorithm that essentially checks all possible candidates 
for the multigraph $H$ and generates a plain text full proof of the lemma.
The source code and the full proof can be found at \cite{computing1}.
Here, we explain the logic behind the algorithm and show that, indeed, it proves 
the lemma.

The general approach of the algorithm is to start with a small graph
and test all relevant extensions, step by step, until it is clear that
no multigraph obtained by further extensions can possibly satisfy the lemma.
Given a multigraph $M$ and sets $U \subset V(M)$ and $R \subset E_S(M)$,
we consider two types of extensions denoted by $\Aa(M,U)$ and $\Mm(M,R)$ as follows:
\begin{itemize}
\item
  $\Aa(M,U)$ is the family of all possible multigraphs $M^+$ obtained from $M$ 
  by adding a vertex $v$ such that $v$~is incident with no multiedge in $M^+$ 
  and $|N_{M^+}(v) \cap U| \geq 1$, and if $\d_{M^+}(v) = 1$ then $v$ has at 
  most one twin in $M^+$,
\item
  $\Mm(M,R)$ is the family of all possible multigraphs $M^+$ obtained from $M$
  by multiplying an edge of $R$.
\end{itemize}
We write $\Aa(M)$ as a short for $\Aa(M,V(M))$, and similarly $\Mm(M)$ for 
$\Mm(M,E_S(M))$.   

For every $k \in \{7, \dots, 10\}$, the computer considers all graphs on $k$ 
vertices containing $C_k$ (but containing no $C_\ell$ subgraph where 
$7 \leq \ell \leq k - 1$).
For each of these graphs, it tests all relevant extensions by calling 
\textproc{investigate}(), see Algorithm 1.

\begin{algorithm}
\caption{Recurrent investigation of $M$, the algorithm branches on 
         $\Bb \subset \Aa(M) \cup \Mm(M)$}
\begin{algorithmic}[1]
\Procedure{investigate}{$M$}
    \If{$L(M)$ is $\Gamma_3$-free \AND $M$ contains no subgraph $D$ and no flat 
           subgraph $C_9^M$}
        \If{conditions (6), \dots, (14) are all satisfied}
        	\State add $M$ to \texttt{F.txt}
        	\State set $\Bb = \Aa(M) \cup \Mm(M)$    
		\Else
        	\State choose a violated condition $p$ from (6), \dots, (14) at random 
        	\State set $\Bb = $ \Call{get\_all\_solution\_attempts}{$M, p$}
        	\Comment{see Algorithm 2}
		\EndIf
		\For{each $M^+ \in \Bb$}
			\State \Call{investigate}{$M^+$}		
			\Comment{recurrence on extended multigraphs}
		\EndFor
	\EndIf
\EndProcedure
\end{algorithmic}
\end{algorithm}

\begin{algorithm}
\caption{Choose a particular problem of type $p$ at random and return all relevant solution attempts}
\begin{algorithmic}[1]
\Procedure{get\_all\_solution\_attempts}{$M, p$}
    \If{$p$ is (6)} 
		\State set $R$ as the set of all edges from $E_S(M)$ incident with a vertex of degree $2$ in $M$
		\State \Return{$\Aa(M) \cup \Mm(M, R)$}
	\EndIf
	\If{$p$ is (7)}	
		\State choose an essential $2$-edge-cut $\{e,f\}$ 
		\If{$\{e,f\} \subset E_S(M)$}	
			\State choose a non-trivial component $C$ of $M-\{e,f\}$
			\State \Return{$\Aa(M, V(C)) \cup \Mm(M, \{e,f\})$}			
		\Else \Comment{$M$ has an essential cutvertex}
			\State reset $p = $ (8)
		\EndIf	
	\EndIf	
	\If{$p$ is (8)}	
		\State choose an essential cutvertex $u$ and a non-trivial component $C$ of $M-u$
		\State \Return{$\Aa(M, V(C))$}
	\EndIf
	\If{$p$ is (9)}	
		\State choose a vertex $u$ of a $K_3$ subgraph such that $|N_M(u)| = 2$
		\State \Return{$\Aa(M, \{u\})$}
	\EndIf
	\If{$p$ is (10)} 
		\State choose a vertex $u$ of a $D^1$ subgraph such that $|N_M(u)| = 2$
		\State \Return{$\Aa(M, \{u\})$}
	\EndIf
	\If{$p$ is (11)} 
		\State choose a $K_{2,4}$ subgraph such that $\sum_{i=1}^4 |N_M(w_i)| = 8$
		\State \Return{$\Aa(M, \{w_1, \dots, w_4\})$}
	\EndIf
	\If{$p$ is (12)} 
		\State choose a $D^2$ subgraph such that $|N_M(c_1)| + |N_M(c_2)| = 4$
		\Comment{or $d_1, d_2$ by symmetry}
		\State \Return{$\Aa(M, \{c_1, c_2\})$}
	\EndIf
	\If{$p$ is (13)} 
		\State choose a subgraph $F \in \mathbb{K}_{4, P_4}^{M}$ such that
		$\sum_{u \in V(F)\sm \{v_1,v_2\}} |N_M(u)| = 16$
		\State \Return{$\Aa(M, V(F)\sm \{v_1,v_2\})$}
	\EndIf
	\If{$p$ is (14)} 
		\State set $U = \bigcup_{i=1}^4 N_M(w_i) \sm \{v_1,v_2\}$
		\State choose a subgraph $F \in \mathbb{K}_{2,4}^M$ such that each vertex of 
               $U$ has degree $1$ in $M$
		\State \Return{$\Aa(M, U\cup\{w_1,\ldots,w_4\})$}
    \EndIf
\EndProcedure
\end{algorithmic}
\end{algorithm}

We should note that considering all extensions of $\Aa(M) \cup \Mm(M)$ in each 
iteration is not desirable since it leads to infinite families of multigraphs 
satisfying the condition on line~2 of Algorithm~1 
(and hence the process never finishes).
The key idea of the proof is to choose a particular violation of a condition from 
(6), \dots, (14) and to test just all extensions potentially helping to solve this 
violation (see lines 7 and 8 of Algorithm 1).
It turns out that every branch of this investigation is finite
since at some point each $M^+ \in \Bb$ fails to satisfy the condition on line 2 of Algorithm 1. 
The relevant extensions are obtained by calling \textproc{get\_all\_solution\_attempts}(),
see Algorithm 2.

We now show that the algorithm, indeed, proves the lemma.
It suffices to show that a multigraph $H$ is added to \texttt{F.txt} by the algorithm
if and only if $H$ completely satisfies the hypothesis of the lemma.

We first consider a multigraph $H$ added to \texttt{F.txt}
and we show that it satisfies the hypothesis of the lemma.
We recall that $H$ is obtained from a simple graph on $k$ vertices containing $C_k$
where $k \in \{7, \dots, 10\}$ by iteratively applying extensions $\Aa()$ and $\Mm()$,
and we note that $H$ has the following properties:
\begin{itemize}
\item
   $H$ contains neither loops nor pendant multiedges,
\item
   $H$ satisfies conditions (1), (2) and (3).
\end{itemize}
In particular, we note that pendant simple edges of $H$ precisely correspond to 
simplicial vertices of $L(H)$, and thus $H$ is a preimage of a line graph of 
multigraph.
Since $H$ is added to \texttt{F.txt} at line 
4 of Algorithm~1, $H$ satisfies the conditions on lines 2 and 3 of Algorithm~1.
We conclude that $H$ is a preimage of a $\Gamma_3$-free line graph of multigraph
and $H$ satisfies all conditions (1), \dots, (14). 

Next, we let $H$ be a multigraph which satisfies the hypothesis of the lemma,
and we show that $H$ is added to \texttt{F.txt}.
Since $H$ is the preimage of a line graph of multigraph,
$H$ has neither loops nor pendant multiedges.
We show the following four claims.

%
%
\setcounter{prostrclaim}{0}
\begin{claim} 
\label{claim-lemmaX-1}
Every flat subgraph of $H$ satisfies the condition on line 2 of Algorithm 1.
\end{claim}

%
%
\begin{claim} 
\label{claim-lemmaX-2}
Let $M$ be a flat subgraph of $H$ and let $U \subset V(M)$
such that for every $u \in U$, all vertices of $N_M(u)$ of degree $1$ in $M$
also belong to $U$.
If there is a vertex $x$ of $V(H) \sm V(M)$ such that $|N_H(x) \cap U| \geq 1$,
then $H$ contains a multigraph $M^+$ from $\Aa(M, U)$
as a flat subgraph.
\end{claim}

%
%
\begin{claim} 
\label{claim-lemmaX-3}
Let $M$ be a flat subgraph of $H$.
If $M$ is distinct from $H$,
then $H$ contains a multigraph $M^+$ from $\Aa(M) \cup \Mm(M)$
as a flat subgraph (possibly $M^+ \iso H$).
\end{claim}

%
%
\begin{claim} 
\label{claim-lemmaX-4}
Let $M$ be a flat subgraph of $H$.
If $M$ violates a condition $p$ from (6), \dots, (14),
then $H$ contains some multigraph $M^+$ given by 
{\rm \textproc{get\_all\_solution\_attempts}($M, p$)}
as a flat subgraph.
\end{claim}

\begin{proofclbt}\underline{of Claim~\ref{claim-lemmaX-1}.}
For the sake of a contradiction,
we suppose that $H$ has a flat subgraph $M$ which fails to satisfy the condition on line 2 of Algorithm 1.
Hence, $M$ contains a subgraph $X$ such that $L(X) \iso \Gamma_3$ 
or contains $D$ as a subgraph
or $M$ contains $C_9^M$ as a flat subgraph.
Since $M$ is a flat subgraph of $H$,
every (flat) subgraph of $M$ is also a (flat) subgraph of $H$.
Hence,
$H$ also fails to satisfy the condition on line 2 of Algorithm 1.
In other words, $L(H)$ is not $\Gamma_3$-free
or $H$ violates condition (4) or (5),
a contradiction.
\end{proofclbt}

\begin{proofclbt}\underline{of Claim~\ref{claim-lemmaX-2}.}
We let $M'$ be the multigraph obtained from $M$
by adding a new vertex adjacent by simple edges to precisely the vertices of $N_{H}(x) \cap V(M)$.
Clearly, $M'$ is a flat subgraph of $H$. 
If $M'$ belongs to $\Aa(M, U)$, then we are done.
Hence, we can assume that $M'$ does not belong to $\Aa(M, U)$,
and it follows that the new vertex is of degree $1$ in $M'$ and has more than one twin in $M'$.
We let $A$ be a set consisting of the new vertex and its two twins in $M'$.
Since $H$ satisfies condition (1), there exists a vertex, say $y$, of $V(H) \sm V(M')$
such that $|N_{H}(y) \cap A| \geq 1$.
We choose a vertex $u \in A$ such that $|N_{H}(y) \cap A \sm \{u\}| \geq 1$,
and we consider the multigraph $M' - u$ (we note that it is isomorphic to $M$).
We let $M^+$ be the multigraph obtained from $M' - u$
by adding a new vertex adjacent by simple edges to precisely the vertices of
$N_{H}(y) \cap V(M' - u)$.
Clearly, $M^+$ is a flat subgraph of $H$.
Finally, we note that the hypothesis of Claim~\ref{claim-lemmaX-2} implies that  
$|A \cap U| \geq 2$.
Since the vertices of $A$ are twins in $M'$ 
and the new vertex of $M^+$ is adjacent to a vertex of $A$,
we conclude that $M^+ \in \Aa(M, U)$.
\end{proofclbt}

\begin{proofclbt}\underline{of Claim~\ref{claim-lemmaX-3}.}
We discuss two cases based on $V(M)$.
For the first case, we suppose that $V(M) = V(H)$.
We recall that $H$ has no loops, no pendant multiedges and no multiedges of multiplicity greater than $2$.
Since $M$ is a flat subgraph of $H$ and $M$ is distinct from $H$,
there is a simple non-pendant edge $e$ of $M$ which corresponds to a multiedge in $H$.
We let $M^+$ be the multigraph obtained from $M$ by multiplying $e$,
and we conclude that $M^+$ is a flat subgraph of $H$ and $M^+$ belongs to $\Mm(M)$.

For the second case, we suppose that $V(M)$ is a proper subset of $V(H)$.
Since $H$ is connected, there is a vertex $x$ of $V(H) \sm V(M)$ such that $|N_H(x)\cap V(M)|\geq 1$.
We observe that we can apply Claim~\ref{claim-lemmaX-2}
with $U = V(M)$
and obtain a desired multigraph $M^+$.
\end{proofclbt}

\begin{proofclbt}\underline{of Claim~\ref{claim-lemmaX-4}.}
We discuss nine cases based on $p$.

First, we consider the case where $p$ is (6) which stands for the fact that the 
multigraph $M$ has few vertices of degree at least~$3$.
We use that $H$ is connected, has no pendant multiedge
and satisfies condition (6) and that $M$ is a flat subgraph of $H$,
and we observe that at least one of the following is true: 
\begin{itemize}
\item
   some vertex of $V(H) \sm V(M)$ is adjacent to a vertex of $V(M)$ in $H$, or
\item
   some edge of $R$ corresponds to a multiedge in $H$. 
\end{itemize}
For the first item, we use Claim~\ref{claim-lemmaX-2} with $U = V(M)$ and 
observe that $H$ contains a multigraph from $\Aa(M)$ 
as a flat subgraph.
For the second item, we note that $H$ contains a multigraph from $\Mm(M, R)$ as a flat subgraph.
The obtained multigraph from $\Aa(M) \cup \Mm(M, R)$ is included at line 4 of Algorithm 2.

For the case (7),
we consider the condition on line 7 of Algorithm 2 and discuss the two options.
We first suppose that the condition is satisfied, that is, $\{e, f\} \subset E_S(M)$.
Since $H$ satisfies condition (7) and $M$ is a flat subgraph of $H$,
at least one of the following is true: 
\begin{itemize}
\item
   some vertex of $V(H) \sm V(M)$ is adjacent to a vertex of $C$, or
\item
   some edge of $\{e, f\}$ corresponds to a multiedge in $H$. 
\end{itemize}
For the first item, we consider an arbitrary vertex $u$ of $C$
and note that all vertices of $N_M(u)$ of degree $1$ in $M$ also belong to $C$.
Hence, we can apply Claim~\ref{claim-lemmaX-2} with $U = V(C)$.
For the second item, we note that $H$ contains a multigraph from $\Mm(M, \{e,f\})$ as a 
flat subgraph.
It follows that $H$ contains a multigraph from 
$\Aa(M, V(C)) \cup \Mm(M, \{e,f\})$ 
as a flat subgraph,
and this multigraph is included at line 9 of Algorithm 2.

Next, we suppose that one of the edges, say $e$, is a pendant edge in $M$.
We observe that at least one of the vertices incident with $f$ is an essential 
cutvertex in $M$.  
Hence, the algorithm can reset $p = (8)$ and continue with lines 12, 13 and 14.

For the case (8),
we use that $H$ satisfies condition (8) and $M$ is a flat subgraph of $H$,
and we note that some vertex of $V(H) \sm V(M)$ is adjacent to a vertex of $C$.
We apply Claim~\ref{claim-lemmaX-2} with $U = V(C)$,
and we conclude that $H$ contains a multigraph from 
$\Aa(M, V(C))$ as a flat subgraph,
and it is included at line 14 of Algorithm 2.
 
For the case (9),
we use that $H$ satisfies condition (9),
and hence some vertex of $V(H) \sm V(M)$ is adjacent to $u$.
Since $|N_M(u)| = 2$, the vertex $u$ has no neighbor of degree $1$ in $M$.
Thus, we can apply Claim~\ref{claim-lemmaX-2} with $U = \{u\}$ and conclude that  
$H$ contains a multigraph from $\Aa(M, \{u\})$ as a flat subgraph. 
This multigraph is included at line 17 of Algorithm 2.

We note that the cases (10), \dots, (13) are similar to (9).
In each case, $M$ contains a set, say~$U$, of vertices (possibly of size $1$)
such that no vertex of $U$ has neighbor of degree $1$ in $M$
and at least one vertex of $U$ has an additional neighbor in $H$.
Hence, Claim~\ref{claim-lemmaX-2}  yields that $H$ contains a multigraph from 
$\Aa(M, U)$ as a flat subgraph, and it is included at the respective line of 
Algorithm~2. 

Lastly, for the case (14) we note that $H$ has a vertex 
of degree at least $2$ adjacent to at least one of the vertices $w_1, \dots, w_4$.
In particular, this vertex has at least two neighbors in $H$ since $H$ has no 
pendant multiedge.
Hence, there exists a vertex $x$ of $V(H) \sm V(M)$ such that at least one of the 
following is true:
\begin{itemize}
\item
   $x$ is adjacent to at least one of $w_1, \dots, w_4$, or
\item
   $x$ is adjacent to a vertex $u$ such that 
   $\d_M(u) = |N_M(u) \cap \{w_1, \dots, w_4\}| = 1$ (in particular, adding 
   $x$ increases the degree of $u$ to $2$).
\end{itemize}
In other words,
$x$ is adjacent to a vertex of $U = \bigcup_{i=1}^4 N_M[w_i] \sm \{v_1,v_2\}$
since each vertex of $N_M(w_i) \sm \{v_1,v_2\}$ has degree $1$ in $M$.
Hence, Claim~\ref{claim-lemmaX-2} yields that $H$ contains a multigraph from 
$\Aa(M, U)$ as a flat subgraph,
and this multigraph is included at line 33 of Algorithm 2.
\end{proofclbt}

With Claims~\ref{claim-lemmaX-1}, \ref{claim-lemmaX-3} and~\ref{claim-lemmaX-4}
on hand, we now show the desired implication.
For the sake of a contradiction,
we suppose that there exists a multigraph $H$ which satisfies the hypothesis of the lemma
but is not added to \texttt{F.txt}. 
We consider all flat subgraphs $M$ of $H$ for which the algorithm calls \textproc{investigate}($M$),
and we choose such a multigraph $M$ maximizing $\sum_{u \in V(M)} \d_M(u)$.
We should also say that such $M$ clearly exists since $H$ satisfies condition (3) and 
the algorithm calls \textproc{investigate}() for the graphs on $k$ vertices containing $C_k$
where $k \in \{7, \dots, 10\}$.
Furthermore,
since $M$ is obtained by recurrently extending one of these graphs by $\Aa()$ and $\Mm()$,
we note that $M$ has no pendant multiedges.

Since $M$ is a flat subgraph of $H$,
$M$ satisfies the condition on line 2 of Algorithm 1 by Claim~\ref{claim-lemmaX-1}.
We now discuss $M$ subject to the condition on line 3 of Algorithm 1
and we obtain a multigraph $M^+$ as follows.
If $M$ satisfies this condition,
then $M$ is added to \texttt{F.txt} at line 4 of Algorithm 1.
Hence, $M$ is distinct from $H$ (since $H$ is not added to \texttt{F.txt}),
and thus we can apply Claim~\ref{claim-lemmaX-3} to $M$; and we let $M^+$ be 
the obtained multigraph.
Otherwise, $M$ violates a condition from (6), \dots, (14), and we can apply 
Claim~\ref{claim-lemmaX-4}  to~$M$; and we let $M^+$ be the obtained multigraph.

We consider the obtained multigraph $M^+$,
and we note that $M^+$ belongs to $\Bb$ due to line 5 or 8 of Algorithm~1.
Finally, line 10 of Algorithm 1 calls \textproc{investigate}($M^+$),
which contradicts the choice of $M$.
Thus, every multigraph satisfying the hypothesis of the lemma
is added to \texttt{F.txt} which concludes the proof of the equivalence.

In order to improve runtime,
our implementation of the algorithm is slightly more involved. 
In the remainder of the proof, we outline details of the implementation
(an interested reader is also invited to have a look at the code and the commentary therein).

In the implementation, we keep track of solved cases
(that is, we save multigraphs whose all extension branches are finished).
Later, when the computer investigates a different case, it tests 
whether this is already solved and then perhaps not investigate it again 
(it tests whether some of the saved multigraphs
appears as a flat subgraph of the multigraph currently investigated).
The list of solved multigraphs can be easily kept short and apt in the recurrence scheme. 

In each branch, we keep track of the extensions.
In particular, an extension is not investigated again if it is known to 
lead to violating the condition on line 2 of Algorithm 1 
or to a case already solved. 

Lastly, a violated condition $p$ and a particular violation are chosen at random,
but the choice favors small $|\Bb|$
(it is not uniformly random, and we consider only violations whose sets $\Bb$ are inclusion minimal).  
In fact, if there is a particular violation such that $|\Bb| = 0$,
then this choice is always preferred
(this means that the violation cannot be fixed 
even with further extensions and the branch finishes).
\end{proof}

\bs

\begin{proofbt} {\bf of Theorem~\ref{thm-main}.} \quad
Let, to the contrary, $G$ be a minimal 3-connected $\{\claw,\Gt\}$-free graph 
that is not Hamilton-connected, let $\bG$ be one of its $\Gt$-closures, and let 
$H=\Lp(\bG)$. Obviously, $H$ is essentially 3-edge-connected since $G$ is 
3-connected. If $H$ has a cutvertex, we can apply our considerations to each 
of its (nontrivial) blocks, hence we can assume that $H$ is essentially 
2-connected. Finally, it is straightforward to verify that if $H\iso W$ 
(see Fig.~\ref{fig-Wagner}), then $G$ is Hamilton-connected, hence $H\niso W$. 
By Theorem~\ref{thmA-strongly_trailable}$(i)$, $H$ contains a cycle of length 
at least 9.

Moreover, we have the following fact.

%
%
\setcounter{prostrclaim}{0}
\begin{claim}
\label{claim-thm1-wagnerace1}
If $\co(H)\in\{W\cup\dvW\}$, then $H$ contains as a subgraph an $\Lp(\Gt)$, or 
$G=L(H)$ is Hamilton-connected. 
\end{claim}

\begin{proofcl}
Set $\tH=\co(H)$.
First observe that if $\tH\in\{W\cup\dvW\}$, then, by Lemma~\ref{lemma-minimality}$(ii)$,
the multiple edge of $\tH$ is of multiplicity 2, hence $\tH\in\{W_1,W_2\}$, where 
$W_1$ and $W_2$ are the multigraphs shown in Fig.~\ref{fig-Wagner-1-2}
(in which $v$ denotes the only vertex of degree 4).

%
%
\begin{figure}[ht]
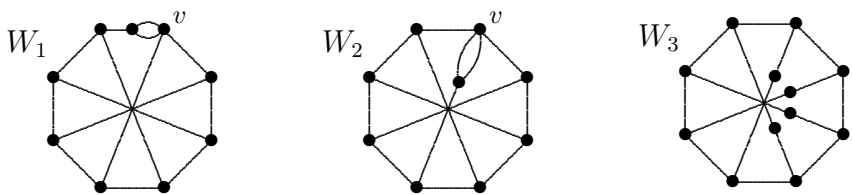

$$\beginpicture
\setcoordinatesystem units <0.6mm,1mm>
\setplotarea x from -70 to 70, y from -5 to 5
\put {\bp
\setcoordinatesystem units <0.35mm,0.35mm>
  \put{$\bullet$} at   -30   12
  \put{$\bullet$} at   -30  -12
  \put{$\bullet$} at    30   12
  \put{$\bullet$} at    30  -12
  \put{$\bullet$} at   -12   30
  \put{$\bullet$} at    12   30
  \put{\footnotesize $v$} at   18  35
  \put{$\bullet$} at   -12  -30
  \put{$\bullet$} at    12  -30
  \put{$\bullet$} at     0   30
  \plot -30 -12 -30 12 -12 30 0 30 /
  \plot 12 30 30 12 30 -12 12 -30
        -12 -30 -30 -12 /
  \plot -30  12  30 -12 /
  \plot -30 -12  30  12 /
  \plot -12  30  12 -30 /
  \plot -12 -30  12  30 /
  \setquadratic
  \plot 0 30  6 33  12 30 /
  \plot 0 30  6 27  12 30 /
  \setlinear
 \put{$W_1$}   at  -40    25
\ep} at -70 0
\put {\bp
\setcoordinatesystem units <0.35mm,0.35mm>
  \put{$\bullet$} at   -30   12
  \put{$\bullet$} at   -30  -12
  \put{$\bullet$} at    30   12
  \put{$\bullet$} at    30  -12
  \put{$\bullet$} at   -12   30
  \put{$\bullet$} at    12   30
  \put{\footnotesize $v$} at   18  35
  \put{$\bullet$} at   -12  -30
  \put{$\bullet$} at    12  -30
  \put{$\bullet$} at     4   10
  \plot -30 -12 -30 12 -12 30 12 30 30 12 30 -12 12 -30
        -12 -30 -30 -12 /
  \plot -30  12  30 -12 /
  \plot -30 -12  30  12 /
  \plot -12  30  12 -30 /
  \plot -12 -30   4  10 /
  \setquadratic
  \plot 4 10  5 21  12 30 /
  \plot 4 10 11 19  12 30 /
  \setlinear
 \put{$W_2$}   at  -40    25
\ep} at   0 0
\put {\bp
\setcoordinatesystem units <0.35mm,0.35mm>
  \put{$\bullet$} at   -30   12
  \put{$\bullet$} at   -30  -12
  \put{$\bullet$} at    30   12
  \put{$\bullet$} at    30  -12
  \put{$\bullet$} at   -12   30
  \put{$\bullet$} at    12   30
  \put{$\bullet$} at   -12  -30
  \put{$\bullet$} at    12  -30
  \put{$\bullet$} at    4   10
  \put{$\bullet$} at   10    4
  \put{$\bullet$} at   10   -4
  \put{$\bullet$} at    4  -10
  \plot -30 -12 -30 12 -12 30 12 30 /
  \plot 12 30 30 12 30 -12 12 -30
        -12 -30 -30 -12 /
  \plot -30  12  30 -12 /
  \plot -30 -12  30  12 /
  \plot -12  30  12 -30 /
  \plot -12 -30  12  30 /
 \put{$W_3$}   at  -40    25
\ep} at  70 0
\endpicture$$
\vspace*{-6mm}
\caption{The multigraphs $W_1$ and $W_2$, and the graph $W_3$.}
\label{fig-Wagner-1-2}
\vspace*{-2mm}
\end{figure}

Let thus $\tH\in\{W,W_1,W_2\}$, and let $V_1$ be the set of all vertices of $\tH$
that are not incident to a double edge. 

It is easy to see that if a vertex $x\in V_1$ is incident in $H$ to neither a 
pendant edge nor an edge containing a vertex of degree 2, then $x$ is not 
necessarily visited by an IDT in $H$ (all edges incident to $x$ can be dominated 
without visiting $x$), and then it is straightforward to verify that $H$ has an
$(e,f)$-IDT for any $e,f\in E(H)$, a contradiction. (Note that this fact can be 
also alternatively seen by applying Lemma~5 from \cite{LRVXY23-I} to $\tH$ and 
setting $A=V_1\sm\{x\}$ if $\tH\iso W$, or 
$A=V_1\cup\{v\}\sm\{x\}$ if $\tH\in\{W_1,W_2\}$).

Thus, in $H$, every vertex in $V_1$ is incident to a pendant edge or to an edge 
containing a vertex of degree 2. To prove the claim, it suffices to show that 
every such multigraph either contains an $\Lp(\Gt)$, or its line graph is 
Hamilton-connected. This will be done in the next claim, 
since it is straightforward to verify that $L(W_3)$ (where $W_3$ is 
the graph shown in Fig.~\ref{fig-Wagner-1-2}) is Hamilton-connected.
\end{proofcl}

%
%
%
\begin{claim}
\label{claim-thm1-wagnerace2}
Let $M\in\{W,W_1,W_2\}$, let $E\subset E(M)$, and let $V$ be the set of all 
vertices from $V(M)$ that are incident to neither an edge from $E$ nor to a
multiedge in $M$. Let $N$ be the multigraph obtained from $M$ by subdividing each 
edge of $E$ with one vertex of degree 2 and by adding one pendant edge to each 
vertex of $V$. 
If $N$ does not contain a subgraph isomorphic to $\Lp(\Gt)$, then $N\iso W_3$.
\end{claim}

\begin{proofcl}
We prove the lemma with the help of a computer. 
For each choice of $M$ and $E$, we test the obtained multigraph $N$.
If $N$ does not contain $\Lp(\Gt)$ as a subgraph, then the program outputs $N$.
The only such $N$ is the graph $W_3$.
The source code of the proof can be found at \cite{computing2}.
\end{proofcl}

Thus, by Theorem~\ref{thmA-strongly_trailable}$(ii)$ and 
Theorem~\ref{thmA-core}$(v)$, we have $|V(\co(H))|\geq 10$, i.e., $H$ has at 
least 10~vertices of degree at least~3.

\ms

Let now $C$ be a cycle in $H$ such that $C$ is a shortest cycle of length at 
least 7.

\begin{mylist}
  \item[\underline{\bf Case 1:}] {\sl $7\leq|V(C)|\leq 8$, or 
  $9\leq|V(C)|\leq 10$ and $C$ is not chordless.} \quad \\
  By the assumptions and by Lemmas~\ref{lemma-diamonds} and 
  \ref{lemma-minimality}, $H$ satisfies the assumptions of Lemma~\ref{lemmaX}.
  Thus, by Lemma~\ref{lemmaX}, $H\in\cF$. To reach a contradiction, it remains 
  to show that for each $H\in\cF$, $L(H)$ is Hamilton-connected. 
  This was done with the help of a computer. 
  For each $H\in\cF$, and for each pair of vertices of $L(H)$, we find a hamiltonian 
  path by using the function `hamiltonian$\_$path' in SageMath. As a certificate of 
  Hamilton-connectedness, we also provide a list of these hamiltonian paths along 
  with a simple program that verifies this certificate. The verification is faster 
  (less than a minute) than finding the hamiltonian paths (a few hours). 
  The source codes are available at \cite{computing1}.

  \item[\underline{\bf Case 2:}] {\sl $9\leq|V(C)|\leq 10$ and $C$ is chordless,
  or $|V(C)|\geq 11$.} \quad \\
  Note that, by the choice of $C$ and by Case~1, $H$ contains no cycle $C'$ 
  of length $7\leq|V(C')|<|V(C)|$. 
  Denote $|V(C)|=r$, $V(C)=\{x_0,x_1,\ldots,x_{r-1}\}$, and $R=V(H)\sm V(C)$.

  %
  %
  \begin{claim}
  \label{claim-thm1-3}
  The cycle $C$ is chordless and for any vertices $x_i,x_j\in V(C)$ with 
  $\dist_C(x_i,x_j)\geq 3$, $N_H(x_i)\cap N_H(x_j)=\emptyset$.
  \end{claim}

\bsm

  \begin{proofcl} For $9\leq r\leq 10$, $C$ is chordless by the assumption of 
  the case.
  If $r\geq 11$ and $x_ix_j$ is a chord in $C$ (i.e., $\dist_C(x_i,x_j)\geq 2$
  and $x_ix_j\in E(H)$), then the edge $x_ix_j$ creates with one of the two 
  parts of $C$ joining $x_i$ and $x_j$ a cycle $C'$ of length 
  $7\leq |V(C')|\leq r-1$, a contradiction. Thus, $C$ is chordless.

  If $r\geq 9$, $\dist_C(x_i,x_j)\geq 3$ and $x_i,x_j$ have a common neighbor 
  $y\in R$, then similarly the path $x_iyx_j$ creates with one of the two parts 
  of $C$ joining $x_i$ and $x_j$ a cycle $C'$ of length 
  $7\leq |V(C')|\leq |V(C)|-1$, a contradiction.
  \end{proofcl}

  The next several claims will be proved in a slightly more general setting 
  for a cycle in $H$ satisfying the conditions given in Claim~\ref{claim-thm1-3}
  (and will be therefore true also for the cycle $C$). 
  Let thus $C'=y_0y_1\ldots y_{t-1}$, $t=|V(C')|\geq 9$, be a cycle in $H$
  such that $C'$ is chordless and $N_H(y_i)\cap N_H(y_j)=\emptyset$ for any 
  $y_i,y_j\in V(C')$ with $\dist_{C'}(y_i,y_j)\geq 3$. Let $R'=V(H)\sm V(C')$.
  If $N_{R'}(y_i)\neq\emptyset$, we will sometimes use $\bar{y_i}$ for (some)
  neighbor of $y_i$ in $R'$.

  %
  %
  \begin{claim}
  \label{claim-thm1-4}
  $R'\neq\emptyset$.
  \end{claim}

\bsm

  \begin{proofcl} 
  Let, to the contrary, $V(C')=V(H)$. If there are two nonconsecutive edges 
  $e_1,e_2\in E(C')\cap E_S(H)$, then $\{e_1,e_2\}$ is an essential edge-cut 
  of $H$ (recall that $C'$ is chordless), contradicting the connectivity 
  assumption. Thus, since $t=|V(C')|\geq 9$, we can choose the notation such 
  that, say, $\{y_0y_1,y_5y_6\}\subset E_M(H)$. Then the edges 
  $(y_0y_1)^1,(y_0y_1)^2,y_1y_2$, $y_2y_3,y_3y_4,y_4y_5,(y_5y_6)^1,(y_5y_6)^2$ 
  determine an $\Lp(\Gt)$ in $H$, a contradiction.
  Thus, $R'=V(H)\sm V(C')\neq\emptyset$.
\bsm  
  \end{proofcl}

\bsm
  
  %
  %
  \begin{claim}
  \label{claim-thm1-5}
  $E(C')\cap E_M(H)\neq\emptyset$.
  \end{claim}

\bsm

  \begin{proofcl} 
  Let, to the contrary, $E(C')\subset E_S(H)$. Since $R'\neq\emptyset$, we can 
  choose the notation such that $N_{R'}(y_0)\neq\emptyset$. Then 
  $N_{R'}(y_4)=\emptyset$ for otherwise the edges 
  $y_{t-1}y_0,\bar{y_0}y_0,y_0y_1,y_1y_2,$ $y_2y_3$, $y_3y_4,y_4\bar{y_4},y_4y_5$
  determine an $\Lp(\Gt)$ in $H$. By the connectivity assumption, there cannot 
  be two consecutive vertices of degree 2 on $C'$, hence 
  $N_{R'}(y_3)\neq\emptyset$ and $N_{R'}(y_5)\neq\emptyset$.
  Now, if $N_{R'}(y_1)\neq\emptyset$, then the edges 
  $y_0y_1,\bar{y_1}y_1,y_1y_2,y_2y_3,y_3y_4,y_4y_5,y_5\bar{y_5},y_5y_6$
  determine an $\Lp(\Gt)$ in $H$, hence $N_{R'}(y_1)=\emptyset$, implying 
  (by the connectivity assumption) that $N_{R'}(y_2)\neq\emptyset$.
  Then $N_{R'}(y_6)=\emptyset$ for otherwise the edges 
  $y_1y_2,\bar{y_2}y_2,y_2y_3,y_3y_4,y_4y_5,y_5y_6,y_6\bar{y_6},y_6y_7$
  determine an $\Lp(\Gt)$ in $H$. By the connectivity assumption, 
  $N_{R'}(y_7)\neq\emptyset$, but then the edges 
  $y_2y_3,\bar{y_3}y_3,y_3y_4,y_4y_5,y_5y_6,y_6y_7,y_7\bar{y_7},y_7y_8$
  determine an $\Lp(\Gt)$ in $H$, a contradiction.
  \end{proofcl} 

\bsm
 
  Thus, $C'$ contains at least one multiple edge.

  %
  %
  \begin{claim}
  \label{claim-thm1-6}
  $C'$ contains an edge $y_iy_{i+1}$ such that $y_iy_{i+1}\in E_M(H)$ and 
  at least one of $y_i$, $y_{i+1}$ has a neighbor in $R'$.
  \end{claim}

\bsm

  \begin{proofcl} 
  Suppose, to the contrary, that $N_{R'}(y_i)=N_{R'}(y_{i+1})=\emptyset$
  for any multiple edge $y_iy_{i+1}$ of $C'$, and choose the notation such that 
  $y_0y_1\in E_M(H)$. 
  Since $\{y_{t-1}y_0, y_1y_2\}$ cannot be a cutset, by symmetry, we can assume 
  that $y_1y_2\in E_M(H)$. Now $\{y_{t-1}y_0, y_2y_3\}$ cannot be a cutset, 
  implying, by symmetry, $y_2y_3\in E_M(H)$. Repeating the argument, we have 
  $y_5y_6\in E_M(H)$, but then the edges 
  $(y_0y_1)^1,(y_0y_1)^2,y_1y_2,y_2y_3,y_3y_4,y_4y_5,(y_5y_6)^1,(y_5y_6)^2$
  determine an $\Lp(\Gt)$ in $H$, a contradiction.
  \end{proofcl}

\bsm
  
  %
  %
  \begin{claim}
  \label{claim-thm1-7}
  Let $C'=y_0y_1\ldots y_{t-1}$ be a cycle in $H$ such that $t=|V(C')|\geq 9$, 
  $C'$ is chordless, and $N_H(y_i)\cap N_H(y_j)=\emptyset$ for any two vertices 
  $y_i,y_j\in V(C')$ with $\dist_{C'}(y_i,y_j)\geq 3$. 
  Then $t\equiv 0 \pmod{3}$, and the notation can be chosen such that 
  \begin{mathitem}
  \item if $i\equiv 0 \pmod{3}$, then $N_{R'}(y_i)\neq\emptyset$ and 
        $y_iy_{i+1}\in E_M(H)$, and 
  
  \item if $i\equiv 1 \pmod{3}$ or $i\equiv 2 \pmod{3}$, then 
        $N_{R'}(y_i)=\emptyset$ and $y_iy_{i+1}\in E_S(H)$.
  \end{mathitem}
  \end{claim}

\bsm

  Recall that, specifically, the cycle $C$ satisfies the assumptions of 
  Claim~\ref{claim-thm1-7} by Claim~\ref{claim-thm1-3}.

  \begin{proofcl} 
  By Claim~\ref{claim-thm1-6}, choose the notation such that $y_0y_1\in E_M(H)$
  and $N_{R'}(y_0)\neq\emptyset$. Then immediately 
  $N_{R'}(y_4)=N_{R'}(y_5)=\emptyset$ and $\{y_4y_5,y_5y_6\}\subset E_S(H)$
  (since otherwise, in all cases, we have an $\Lp(\Gt)$ in $H$).
  By the connectivity assumption, necessarily $y_3y_4\in E_M(H)$ and 
  $N_{R'}(y_6)\neq\emptyset$ or $y_6y_7\in E_M(H)$. Then necessarily 
  $N_{R'}(y_2)=\emptyset$ and $y_1y_2\in E_S(H)$ (otherwise we have an 
  $\Lp(\Gt)$ in $H$). Since $\{y_1y_2,y_4y_5\}$ is not a cutset, 
  $N_{R'}(y_3)\neq\emptyset$.
  
  Summarizing, we conclude that $N_{R'}(y_3)\neq\emptyset$, $y_3y_4\in E_M(H)$, 
  $N_{R'}(y_4)=N_{R'}(y_5)=\emptyset$, and $\{y_4y_5,y_5y_6\}\subset E_S(H)$, 
  i.e., we have the requested statement for $i=3,4,5$. Repeating the argument, 
  starting with $N_{R'}(x_3)\neq\emptyset$ and $y_3y_4\in E_M(H)$, we get
  the statement for $i=6,7,8$. The claim then follows by induction.
  \end{proofcl} 

\bsm

  Now we can apply Claim~\ref{claim-thm1-7} to the cycle $C$ (recall that, 
  by Claim~\ref{claim-thm1-3}, $C$ satisfies the assumptions of 
  Claim~\ref{claim-thm1-7}). 
  Thus, we have $|V(C)|= r \equiv 0 \pmod{3}$, which specifically implies that
  $r=9$ or $r\geq 12$. 
  
  For $\ell=0,1,2$, we denote 
  $X_\ell=\{x_i\in V(C)|\ i \equiv \ell \pmod{3}\}$ and 
  $E_\ell=\{x_ix_{i+1}\in E(C)|\ i \equiv \ell \pmod{3}\}$.
  Thus, specifically, by Claim~\ref{claim-thm1-7}, $E_0\subset E_M(H)$ and 
  $E_1\cup E_2\subset E_S(H)$. 
  By the connectivity assumption, the edges in $E_1\cup E_2$ cannot form a
  cutset of $H$, hence there is an $(x_\alpha,x_\beta)$-path with endvertices 
  $x_\alpha,x_\beta\in X_0$ and with interior vertices in $R$.
  
  For the purpose of this proof, by an {\em arc} we will mean a path that is 
  chordless and any two vertices at distance at least 3 have no common neighbor. 
  Clearly, if for some vertices $x_\alpha,x_\beta\in X_0$ there is an 
  $(x_\alpha,x_\beta)$-path with interior vertices in $R$, there is also 
  an $(x_\alpha,x_\beta)$-arc with interior vertices in $R$.

  Let $T$ be an $(x_\alpha,x_\beta)$-arc with interior vertices in $R$ for 
  some $x_\alpha,x_\beta\in X_0$. Since for any two vertices 
  $x_\alpha,x_\beta\in X_0$, $\dist_C(x_\alpha,x_\beta)\geq 3$ and $C$ is 
  shortest, $T$ is of length at least 3.
  The vertices $x_\alpha,x_\beta$ divide $C$ into two paths (called 
  {\em segments}), and each of them, together with $T$, creates in $H$ a 
  cycle of length at least 6. Let $C^T$ be the shorter one of these cycles 
  (or any of them if both segments of $C$ are of the same length).
  
  \ms
  
  Suppose first that the $(x_\alpha,x_\beta)$-arc $T$ can be chosen such that 
  $|V(C^T)|\geq 7$, and, subject to this condition, choose $T$ such that $C^T$
  is shortest possible. By the choice of $C$ and by the previous arguments, 
  $9\leq|V(C)|\leq|V(C^T)|$. Choose the notation such that, in the natural 
  orientation of $C$ given by increasing indices, the segment of $C$ that is 
  in $C^T$ is oriented from $x_\alpha$ to $x_\beta$, and set 
  $T=z_0z_1\ldots z_t$, where $z_0=x_\alpha$ and $z_t=x_\beta$. 
  
  We verify that the cycle $C^T$ satisfies the assumptions of 
  Claim~\ref{claim-thm1-7}. This is clear for $y_i,y_j\in V(C)$ 
  (by Claim~\ref{claim-thm1-3}), and for $y_i,y_j\in V(T)$ (by the choice of $T$);
  so, let $y_i\in V(C^T)\sm V(T)$ and $y_j\in V(C^T)\sm V(C)$. 
  Since $y_i\in X_0$ by Claim~\ref{claim-thm1-7}, $\dist_C(y_i,x_\alpha)\geq 3$ 
  and $\dist_C(y_i,x_\beta)\geq 3$, hence $\dist_C(x_\alpha,x_\beta)\geq 6$
  (by the choice of $C^T$ as the shorter one determined by 
  $x_\alpha,x_\beta$ and $T$), implying $|V(C^T)|\geq 12$. Thus, any chord 
  $y_iy_j\in E(H)$ or a vertex in $N_R(y_i)\cap N_R(y_j)$ yields a cycle 
  $C'$ of length $7\leq|V(C')|<|V(C^T)|$, contradicting the choice of $C^T$.
  Thus, $C^T$ satisfies the assumptions of Claim~\ref{claim-thm1-7}.
  This implies that $t \equiv 0 \pmod{3}$ and $z_2z_3\in E_M(H)$. 
  Since also $x_{\alpha-3}x_{\alpha-2}\in E_M(H)$, the edges 
  $(x_{\alpha-3}x_{\alpha-2})^1,(x_{\alpha-3}x_{\alpha-2})^2,
  x_{\alpha-2}x_{\alpha-1},x_{\alpha-1}x_\alpha,x_\alpha z_1$, $z_1z_2,
  (z_2z_3)^1,(z_2z_3)^2$ determine an $\Lp(\Gt)$ in $H$, a contradiction.

  \ms

  Thus, for every $(x_\alpha,x_\beta)$-arc $T$ in $H$ with interior vertices 
  in $R$, we have $|V(C^T)|=6$, i.e., $t=3$ and $\dist_C(x_\alpha,x_\beta)=3$.
  Then, replacing in $C$ the $(x_\alpha,x_\beta)$-segment (of length 3) 
  by the $(x_\alpha,x_\beta)$-arc $T$, we get a cycle $C'$ of the same 
  length as $C$, and, by Claim~\ref{claim-thm1-7}, we have $z_0z_1\in E_M(H)$, 
  $\{z_1z_2,z_2z_3\}\subset E_S(H)$, and $N_{R'}(z_1)=N_{R'}(z_2)=\emptyset$
  (where $R'=V(H)\sm V(C')$).
  Note that, by the connectivity assumption, there is at most one 
  $(x_\alpha,x_\beta)$-segment of $C$ of length 3 without any 
  $(x_\alpha,x_\beta)$-arc with interior vertices in $R$. Then it is 
  straightforward to verify that $H$ has an $(e,f)$-IDT for any $e,f\in E(H)$,
  hence $G^{\Gt}=L(H)$ is Hamilton-connected by Theorem~\ref{thmA-DCT+IDT}
  (for an example, see Fig.~\ref{fig-IDT}$(a)$). 
  Alternatively viewed, $L(H)$ is the graph that can be obtained from a cycle
  by replacing each vertex with a clique and each edge, except at most one, 
  with at least 2 vertices of degree 3 attached to both cliques 
  (see Fig.~\ref{fig-IDT}$(b)$), and it is
  straightforward to verify that $G^{\Gt}$ is Hamilton-connected.

%
%
\begin{figure}[ht]
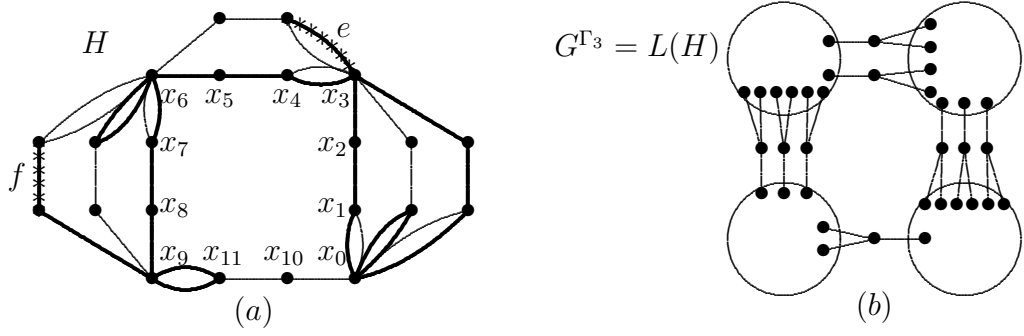

$$\bp
\setcoordinatesystem units <0.9mm,1mm>
\setplotarea x from -50 to 50, y from -7 to 7
\put{\beginpicture
\setcoordinatesystem units <1.5mm,1.5mm>
\setplotarea x from -10 to 10, y from -10 to 10

\put{$\bullet$} at   -9   9
\put{$\bullet$} at   -9   3
\put{$\bullet$} at   -9  -3
\put{$\bullet$} at   -9  -9
\put{$\bullet$} at   -3   9
\put{$\bullet$} at   -3  -9
\put{$\bullet$} at    3   9
\put{$\bullet$} at    3  -9
\put{$\bullet$} at    9   9
\put{$\bullet$} at    9   3
\put{$\bullet$} at    9  -3
\put{$\bullet$} at    9  -9
\put{$\bullet$} at  -14   3
\put{$\bullet$} at  -14  -3
\put{$\bullet$} at  -19   3
\put{$\bullet$} at  -19  -3
\put{$\bullet$} at   14   3
\put{$\bullet$} at   14  -3
\put{$\bullet$} at   19   3
\put{$\bullet$} at   19  -3
\put{$\bullet$} at   -3   14
\put{$\bullet$} at    3   14
\put{$x_0$} at    7   -7
\put{$x_1$} at    7   -2.7
\put{$x_2$} at    7    2.7
\put{$x_3$} at    7.3  7
\put{$x_4$} at    3    7
\put{$x_5$} at   -3    7
\put{$x_6$} at   -7    7
\put{$x_7$} at   -7    2.7
\put{$x_8$} at   -7   -2.7
\put{$x_9$} at   -7   -7
\put{$x_{10}$} at   2.7   -7
\put{$x_{11}$} at  -2.7   -7
\plot -3 -9  9 -9 /
\plot  -9 9  -3 14  3 14  /
\plot  9 9  14 3  14 -3 /  
\plot -9 -9  -14 -3  -14 3 /  
\setquadratic
\plot  3 9    6 9.8   9 9 /
\plot  9 9   5 11   3 14  /
\plot -9 9   -9.7 6  -9 3  /
\plot -9 9  -14 7   -19 3  -14 5    -9 9 /
\plot  9 -9    9.7 -6   9 -3     /
\plot   19 -3   14 -5     9 -9 /
\setplotsymbol ({\large .})
\plot  9 -9   14 -7    19 -3 /
\plot  9 -9   12 -6    14 -3   11.2 -5   9 -9 /
\plot  9 -9    8.3 -6   9 -3    /
\plot  9 9   6 8.2  3 9 /
\plot  3 14  6.5 12   9 9 /   
\plot -9 -9  -6 -8 -3 -9  -6 -10  -9 -9 /
\plot -9 9  -12 6   -14 3  -11.2 5  -9 9 /
\plot -9 3   -8.3 6  -9 9 /
\setlinear
\plot  9 9  19 3  19 -3 /  
\plot 9 -3  9 9 /
\plot 3 9  -9 9 /
\plot -9 -9  -19 -3  -19 3 /  
\plot -9 -9  -9 3 /
\setplotsymbol ({\fiverm .})
%
\setplotsymbol ({$*$})
\plotsymbolspacing=5pt
\plot -19 -3  -19 3 /
\setquadratic
\plot  3 14  6.5 12   9 9 /
\setlinear
\plotsymbolspacing=0.4pt
\setplotsymbol ({\fiverm .})
\put{$f$} at  -21   0 
\put{$e$} at    8    13 
\put{$H$} at   -14   12 
\put{$(a)$} at  0 -12
\endpicture} at  -40   0
\put{\beginpicture
\setcoordinatesystem units <1.5mm,1.5mm>
\setplotarea x from -10 to 10, y from -10 to 10
\put{$\bullet$} at  5  11
\put{$\bullet$} at  5   9
\put{$\bullet$} at  5   7
\put{$\bullet$} at  5   5
\put{$\bullet$} at  0   9.5
\put{$\bullet$} at  0   6.5
\put{$\bullet$} at -4   9.5
\put{$\bullet$} at -4   6.5
\plot 5 11 0 9.5 5 9 /
\plot 5  7 0 6.5 5 5 /
\plot -4 9.5   0 9.5 /
\plot -4 6.5   0 6.5 /
\put{$\bullet$} at -11.5  5
\put{$\bullet$} at -10.1  5
\put{$\bullet$} at  -8.7  5
\put{$\bullet$} at  -7.3  5
\put{$\bullet$} at  -5.9  5
\put{$\bullet$} at  -4.5  5
\put{$\bullet$} at  -8    0
\put{$\bullet$} at -10    0
\put{$\bullet$} at  -6    0
\put{$\bullet$} at  -8   -4
\put{$\bullet$} at -10   -4
\put{$\bullet$} at  -6   -4
\plot -11.5 5  -10 0  -10.1 5 /
\plot  -8.7 5   -8 0   -7.3 5 /
\plot  -5.9 5   -6 0   -4.5 5 /
\plot  -10  0  -10 -4 /
\plot   -8  0   -8 -4 /
\plot   -6  0   -6 -4 /
\put{$\bullet$} at  11.5 -5
\put{$\bullet$} at  10.1 -5
\put{$\bullet$} at   8.7 -5
\put{$\bullet$} at   7.3 -5
\put{$\bullet$} at   5.9 -5
\put{$\bullet$} at   4.5 -5
\put{$\bullet$} at   8    0
\put{$\bullet$} at  10    0
\put{$\bullet$} at   6    0
\put{$\bullet$} at   8    4
\put{$\bullet$} at  10    4
\put{$\bullet$} at   6    4
\plot  11.5 -5   10 0  10.1 -5 /
\plot   8.7 -5    8 0   7.3 -5 /
\plot   5.9 -5    6 0   4.5 -5 /
\plot   10   0   10 4 /
\plot    8   0    8 4 /
\plot    6   0    6 4 /
\put{$\bullet$} at   0   -8
\put{$\bullet$} at   4.5 -8
\put{$\bullet$} at  -4.5 -7
\put{$\bullet$} at  -4.5 -9
\plot -4.5 -9  0 -8  -4.5 -7 /
\plot 0 -8  4.5 -8 /
\circulararc  360  degrees from   3  8 center at  8   8
\circulararc  360  degrees from   3 -8 center at  8  -8
\circulararc  360  degrees from  -3  8 center at -8   8
\circulararc  360  degrees from  -3 -8 center at -8  -8
\put{$G^{\Gt}=L(H)$} at   -21   9
\put{$(b)$} at   0 -14
\endpicture} at  40  0
\ep$$
\bsm
\caption{An $(e,f)$-IDT in $H$ and the graph $G^{\Gt}=L(H)$.}
\label{fig-IDT}
\end{figure}
\end{mylist}
\bsm\bsm
\end{proofbt}

\section{Concluding remarks}
\label{sec-concluding}

{\bf 1.} Theorem~\ref{thm-main} admits a slight extension as follows. 
For $s\geq 0$, a graph $G$ is {\em $s$-Hamilton-connected} if the graph $G-M$ 
is Hamilton-connected for any set $M\subset V(G)$ with $|M|\leq s$. 
Obviously, an $s$-Hamilton-connected graph must be $(s+3)$-connected. 
Since an induced subgraph of a $\{\claw,\Gt\}$-free graph is also 
$\{\claw,\Gt\}$-free, we immediately have the following fact, showing 
that, in $\{\claw,\Gt\}$-free graphs, the obvious necessary condition is 
also sufficient.

%
%
\begin{corollary}
\label{coro-s-HC}
Let $s\geq 0$ be an integer, and let $G$ be a $\{\claw,\Gt\}$-free graph. 
Then $G$ is $s$-Hamilton-connected if and only if $G$ is $(s+3)$-connected.
\end{corollary}

\ms

{\bf 2.} A $\Gt$-closure of a graph $G$, as defined in 
Section~\ref{sec-Gt-closure}, is not unique in general. However, in view 
of Theorem~\ref{thm-main}, it is unique on 3-connected $\{\claw,\Gt\}$-free graphs
since each such graph is Hamilton-connected, hence has complete closure.

\bs

{\bf 3.} We can now update the discussion of potential pairs $X,Y$ of connected
graphs that might imply Hamilton-connectedness of a $3$-connected $\{X,Y\}$-free 
graph, as summarized in~\cite{RV23}.

As shown in \cite{BFHTV02}, up to a symmetry, necessarily $X=\claw$, and, 
summarizing the discussions from \cite{BGHJFW14}, \cite{BFHTV02}, 
\cite{FFRV12} and \cite{LRVXY23-II}, there are the following possibilities 
for $Y$ (see Fig.~\ref{fig-special_graphs}):

\ssk

\begin{mathitem}{\sl
\item $Y=P_i$ with $4\leq i\leq 9$,
\item $Y=Z_i$ with $i\leq 6$, or $Y=Z_7$ for $n=|V(G)|\geq 21$,
\item $Y=B_{i,j}$ with $i+j\leq 7$,
\item $Y=N_{i,j,k}$ with $i+j+k\leq 7$,
\item $Y\in\{\Gamma_1,\Gamma_3\}$, or $Y=\Gamma_5$ for $n=|V(G)|\geq 21$.}
\end{mathitem}

\ms

Best known results in the direction of each of these subgraphs are summarized 
in Theorem~\ref{thmA-known_results}, and we summarize the current status of 
the problem in the following table.

\ms

\begin{tabular}{|c|c|c|c|c|}
\hline
$Y$ & Possible & Best known &  Reference & Open  \\
\hline
$P_i$      & $4\leq i\leq 9$ & $P_9$   & \cite{BGHJFW14} & --- \\
$Z_i$      & $i\leq 6$; $Z_7$ for $n\geq 21$ & $Z_6$; $Z_7$ for $G\niso L(W^+)$  
       & \cite{RV21}      & --- \\
$B_{i,j}$  & $i+j\leq 7$     & $i+j\leq 7$ & \cite{RV23} & ---  \\
$N_{i,j,k}$ &$i+j+k\leq 7$    & $i+j+k\leq 7 $   
            & \cite{LRVXY23-I,LRVXY23-II,LXL21} & --- \\
$\Gamma_i$ & $\Gamma_1$, $\Gamma_3$, $\Gamma_5$ for $n\geq21$ & $\Gamma_1,\Gt$ 
       & \cite{BFHTV02}, this paper & $\Gamma_5$ for $n\geq21$ \\
\hline
\end{tabular}

\bs

\noi
Thus, the only remaining case is the case $Y=\Gamma_5$ for $n\geq 21$ 
(or possibly $\Gamma_5$ for $G\niso L(W^+)$).
We believe that this case is also true, and we think that it could be 
doable with the techniques of this paper; however, the proof would be too 
technical to be reasonably handled  even with the help of a computer. 

\bs

{\bf 4.} 
The source codes of our proof-assisting programs are available at \cite{computing1}
and \cite{computing2}.
The codes are written in Python 3.8 and use functions imported from 
SageMath~9.6. We thank the SageMath community \cite{Sage} for developing a 
valuable open-source mathematical software.

\section{Acknowledgement}
The research was supported by project GA20-09525S of the Czech Science 
Foundation.

\end{document}